\documentclass[11pt]{article}
\usepackage[margin=1.2in]{geometry} 
\usepackage{amsthm, amsmath,amsfonts,amssymb,euscript,hyperref,graphics,color,slashed,mathrsfs}
\usepackage{graphicx}
\usepackage[toc]{appendix}
\usepackage{comment}
\usepackage{import}
\usepackage{tikz}
\usepackage{latexsym}

\def\ws{\slashed{w}}
\def\vs{\slashed{v}}

\def\Th{{\widehat{T}}}

\def\Cb{{\underline{C}}}

\def\ub{\underline{u}}
\def\wb{\underline{w}}

\def\chih{\widehat{\chi}}

\def\Lb{\underline{L}}

\def\tr{\text{tr}}

\def\diff{\mathrm{d}}
\def\angp{ {\p \mkern-10mu / }}

\def\p{\partial}

\newcommand{\norm}[1]{\left\lVert#1\right\rVert}

\newtheorem*{Main Theorem}{Main Theorem}
\newtheorem*{Main Proposition}{Main Proposition}
\newtheorem{theorem}{Theorem}[section]
\newtheorem{lemma}[theorem]{Lemma}

\newtheorem{corollary}[theorem]{Corollary}

\newtheorem{remark}[theorem]{Remark}

\numberwithin{equation}{section}

\begin{document}
	\title {Constructing characteristic initial data for three dimensional compressible Euler equations}

	\author{
		Yu-Xuan WANG\thanks{School of Mathematics and Statistics, Wuhan University, Wuhan, China. wangyuxu20@mails.tsinghua.edu.cn}, 
		Si-Fan YU\thanks{Department of Mathematics, National University of Singapore, Singapore. sifanyu@nus.edu.sg},
		and Pin YU\thanks{Department of Mathematical Sciences, Tsinghua University, Beijing, China. yupin@mail.tsinghua.edu.cn}}


	
	\maketitle

	\begin{abstract}
		This paper resolves the characteristic initial data problem for the three-dimensional compressible Euler equations — an open problem analogous to Christodoulou's characteristic initial value formulation for the vacuum Einstein field equations in general relativity. Within the framework of acoustical geometry, we prove that for any ``initial cone" $C_0\subset \mathcal{D}=[0,T]\times \mathbb{R}^3$ with initial data $(\mathring{\rho},\mathring{v},\mathring{s})$ 
		given at $S_{0,0}=C_0\cap \Sigma_0$,
		arbitrary smooth entropy function $s$ and angular velocity $\slashed{v}$ determine smooth initial data $(\rho,v,s)$ on $C_0$ that render $C_0$ characteristic. Differing from the intersecting-hypersurface case by Speck-Yu \cite{Speck-Yu} and the symmetric reduction case by Lisibach \cite{CharivpLisibach}, our vector field method recursively determines all (including $0$-th) order derivatives of the solution along $C_0$ via transport equations and wave equations. This work provides a complete characteristic data construction for admissible hypersurfaces in the 3D compressible Euler system, introducing useful tools and providing novel aspects for studies of the long-time dynamics of the compressible Euler flow.
	\end{abstract}
	
	\tableofcontents
	\section{Introduction to the characteristic problem for Euler equations}
	
	We study the Euler system (originally introduced in \cite{Euler1757}) in dimension three which governs the motion of an ideal compressible inviscid fluid in three dimensions:
	\begin{equation}\label{eq: Euler in Euler coordinates}
		\begin{cases}
			(\partial_t + v \cdot \nabla) \rho &= -\rho \nabla \cdot v, 
			\\
			(\partial_t + v \cdot \nabla)v &= -\rho^{-1} \nabla p,\\
			(\partial_t + v \cdot \nabla)s &= 0.
		\end{cases}
	\end{equation} 
	The functions $\rho,p,s$ and $v=v^1\frac{\partial}{\partial x^1}+v^2\frac{\partial}{\partial x^2}+v^3\frac{\partial}{\partial x^3}$ are the density, pressure, entropy and velocity of the gas, respectively. They all depend on the spacetime variables $(t,x^1,x^2,x^3)$.
	Throughout the paper, we choose $\rho$, $v$ and $s$ as independent thermodynamical variables. The pressure $p$ is then a function of the state of the density and the pressure, that is, $p(t,x) = p(\rho(t,x),s(t,x))$ and $p$ is smooth with respect to $\rho$ and $s$. The sound speed $c$ is then a positive function in $\rho$ and $s$ given by the formula $c^2= \dfrac{\partial p}{\partial \rho}(\rho,s)>0$.
	
	\subsection{The acoustical metric and the characteristic hypersurfaces for sound waves}
	If we choose $p$, $v$ and $s$ as the independent thermodynamical variables, the Euler equations can be expressed as
	\begin{equation}\label{eq: Euler in Euler coordinates 2}
		\begin{cases}
			&  \frac{\partial v}{\partial t} +v\cdot \nabla v  +\rho^{-1}\nabla p= 0, \\
			& \frac{\partial p}{\partial t} +(v \cdot \nabla)p=-\rho c^2 (\nabla \cdot v),\\
			&  \frac{\partial s}{\partial t} +(v\cdot \nabla) s=0.
		\end{cases}
	\end{equation}
	We regard the above system as a first-order hyperbolic system and we also assume that the solution exists on a spacetime region $\mathcal{D}\subset \mathbb{R}\times \mathbb{R}^3$. We use $(t,x^1,x^2,x^3, \omega, \xi_1,\xi_2,\xi_3)$ as the standard coordinate systems for $T^*\mathcal{D}$. Therefore, the characteristic polynomial for \eqref{eq: Euler in Euler coordinates 2} is a function defined $T^*\mathcal{D}$ and is a fifth-degree homogeneous polynomial in $(\omega, \xi_1,\xi_2,\xi_3)$. More precisely, it is given by
	\[H(t,x;\omega,\xi)=\left(\omega+v(t,x)\cdot \xi\right)^3 \left((\omega+v(t,x)\cdot \xi)^2-c^2|\xi|^2\right).\]
	The system has two different types of characteristic waves. The first type corresponds to the factor $ \omega+v(t,x)\cdot \xi $ of $H(t,x;\omega,\xi)$. It consists of two families of transverse vorticity wave and one family of entropy waves. The second type corresponds to the factor $(\omega+v(t,x)\cdot \xi)^2-c^2|\xi|^2$ of $H(t,x;\omega,\xi)$, which gives the sound waves.  We then define
	\[H_{_{\rm sound}}(t,x;\omega,\xi)= (\omega+v(t,x)\cdot \xi)^2-c^2|\xi|^2.\]
	In the sequel, the word \emph{characteristic} refers to the characteristic hypersurfaces of the sound waves, i.e., the characteristic hypersurfaces defined by the Hamiltonian $H_{_{\rm sound}}(t,x;\omega,\xi)$. 
	
	We also recall the following definition of the characteristic hypersurfaces. Let $C\subset \mathcal{D}$ be a smooth hypersurface. We assume that $C$ is locally defined as a level set of a smooth function $\phi$ on $\mathcal{D}$ so that the following eikonal equation is satisfied
	\begin{equation}\label{eq: Eikonal sound}
		H_{_{\rm sound}}(d\phi)=0.
	\end{equation}
	
	The spacetime region $\mathcal{D}$ can be equipped with the so-called \emph{acoustical metric} $g$:
	\[ g = - c^2 dt^2 +  \sum_{i=1}^3 (dx^i - v^idt)^2.\]
	This is a Lorentzian metric. The inverse metric on $T^*D$ can then be computed as 
	\[g^{-1} = -\frac{1}{c^2} (\partial_t +\sum_{i=1}^3  v^i \partial_i) \otimes (\partial_t +\sum_{i=1}^3  v^i \partial_i)  +\sum_{i=1}^3 \partial_i \otimes \partial_i.\]
	We recall that the classic Hamiltonian for the geodesic flow of $g$ (on the cotangent bundle of $\mathcal{D}$) is then given by 
	\[H_{_{\rm geodesic}}(\omega,\xi;t,x)=g^{\mu\nu}\xi_\mu\xi_\nu.\]
	where $\xi_\mu =(\omega,\xi_1,\xi_2,\xi_3)$ and we use the Einstein summation convention in the above formula. By the explicit formula of $g^{-1}$, we derive that
	\[H_{_{\rm geodesic}}(\omega,\xi;t,x)=|\omega+v(t,x)\cdot \xi|^2-c^2|\xi|^2.\]
	Therefore, $H_{_{\rm geodesic}}(\omega,\xi;t,x)=H_{_{\rm sound}}(\omega,\xi;t,x)$. In particular,  by \eqref{eq: Eikonal sound}, a characteristic hypersurface $C$ for sound waves is then locally defined by a smooth function $\phi$ on $\mathcal{D}$ so that
	\[
	H_{_{\rm geodesic}}(d\phi)=0.
	\]
	We then conclude that the characteristic hypersurfaces for sound waves are indeed null hypersurfaces for the acoustical metric $g$. These hypersurfaces are ruled by null geodesics. The standard geometric techniques used for null hypersurfaces in general relativity, see \cite{C09,C-K93}, can then be adapted to study sound wave characteristic hypersurfaces.
	
	\subsection{The statement of the problem}\label{SS:Statement}
	
	We consider the following domain $\mathcal{D}=[0,T]\times \mathbb{R}^3\subset \mathbb{R}^4$ where $T\in [0,\infty]$.  We use the standard Cartesian coordinate system $(x^0,x^1,x^2,x^3)$ on $\mathcal{D}$. For $t\in [0,T]$,  we use $\Sigma_t$ to denote the constant slice
	\[\Sigma_t=\big\{(x^0,x^1,x^2,x^3)\in \mathcal{D}\big| x^0=t\big\}.\] 
	Let $C_0$ be a truncated cone depicted in the following picture:
	\begin{center}
		\includegraphics[width=3.5in]{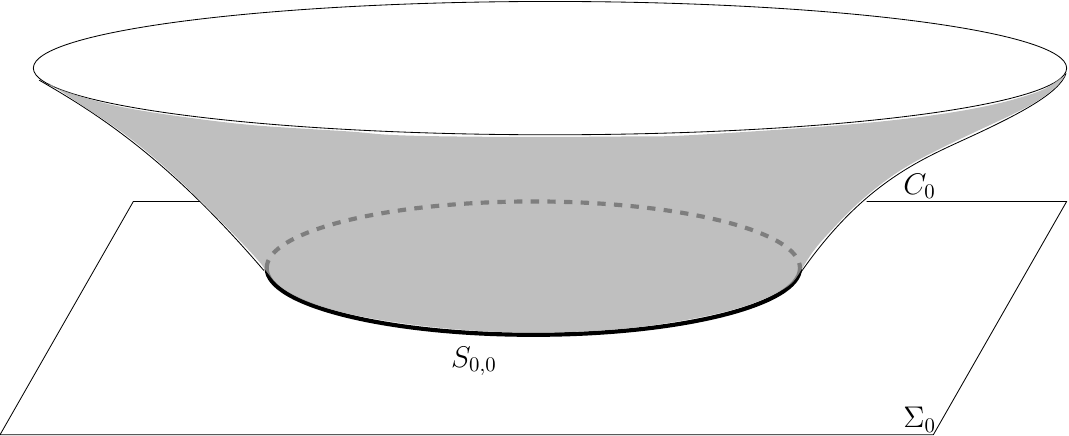}
	\end{center}
	More precisely, a smooth hypersurface $C_0\subset \mathcal{D}$ is called an \emph{initial cone} if it satisfies the following two conditions:
	\begin{itemize}
		\item[1)] For all $t\in [0,T]$,  $C_0$  intersects $\Sigma_t$ transversally;
		\item[2)] For all $t\in [0,T]$, $S_{t,0}:=C_0\cap \Sigma_t$ is a smooth $2$-sphere in $\Sigma_t$.
	\end{itemize}
	The main result of the paper can be stated as follows: 
	\begin{theorem}[Rough version of the main result]\label{TH:RoughMain}
		Given an initial cone  $C_0$ in $\mathcal{D}$, given initial data $(\mathring{\rho},\mathring{v},\mathring{s})$ on $\Sigma_0$ inside $S_{0,0}$,   we can prescribe all possible characteristic initial data $(\rho,v,s)$ on $C_0$, so that 
		\begin{itemize}
			\item[1)]  $(\rho,v,s)$ are smooth functions on $C_0$ ;
			\item[2)] $C_0$ is characteristic hypersurfaces with respect to $H_{_{\rm sound}}(t,x;\omega,\xi)$ for the given $(\rho,v,s)$;
			\item[3)] {$(\rho,v,s)$ agrees with $(\mathring{\rho},\mathring{v},\mathring{s})$ at all orders on $S_{0,0}$.}
		\end{itemize}
	\end{theorem}
	
	\begin{remark}
		Since the proof of the above result is local, we can also change the topological types of $C_0$  to suit other situations. For example, we may take $C_0$ to have the topology of hyperplanes. This can be used to model the gas dynamics near plane symmetry. 
	\end{remark}
	\begin{remark}\label{rem: pure goursat}
		In the theorem, we have already set the initial data on $\Sigma_0$. We can also consider a complete Goursat initial data for two intersection characteristic hypersurfaces $C_0$ and $\Cb_0$ depicted as follows.
		\begin{center}
			\includegraphics[width=3in]{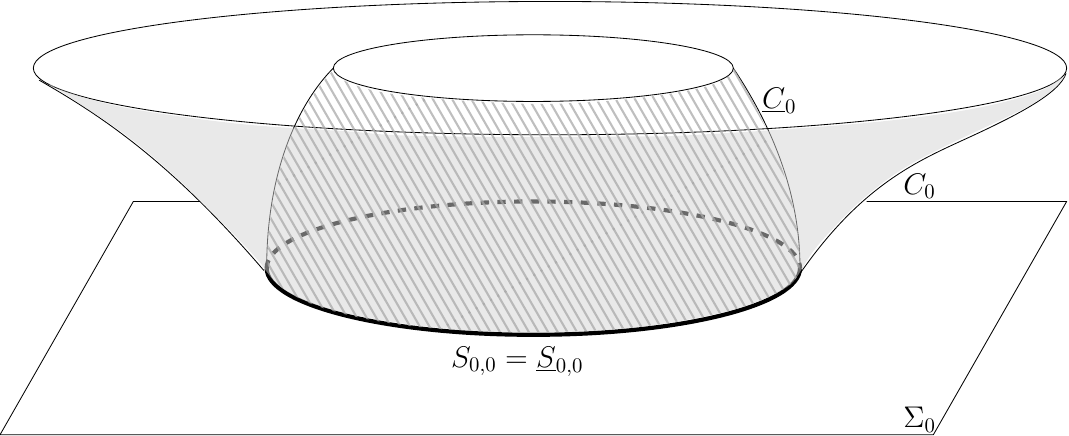}
		\end{center}
		We can also prescribe smooth initial data on $C_0$ and $\Cb_0$ so that $C_0$ and $\Cb_0$ become characteristic hypersurfaces simultaneously. We chose to prove the particular formulation in the above theorem rather than other versions because it presents the clearest and most straightforward expression for such problems while preserving the essential nature of the subject.
	\end{remark}
	\subsection{Motivations}\label{SS:Motivations}
	The study of hyperbolic PDEs on characteristic hypersurfaces has played a central role in various breakthrough results over the last decades (see \cite{C07,C09,C19,C-K93}). The general formulation of the characteristic initial value problem (CharIVP) is expected to be useful for studying solution dynamics. This is either because characteristic initial data offers greater freedom than Cauchy data in certain contexts, or because it arises naturally in specific problems.  
	
	As an example, CharIVP has been fundamental in general relativity for investigating trapped surface formation \cite{C09}.  In studying black hole formation mechanisms, the Cauchy approach requires solving an overdetermined elliptic system – the constraint equations – which presents significant technical difficulties. In contrast, the characteristic approach essentially reduces the initial data specification to solving a system of ordinary differential equations, offering precise control over the resulting initial values.
	
Another important example involves the construction of maximal classical developments in hyperbolic PDEs. Describing this maximal development – the largest possible spacetime region uniquely determined by initial data where classical solutions remain valid – represents a fundamental challenge (see \cite{Abbrescia-SpeckCauchyHorizon, Abbrescia-Speck22, C07, C19} for specific analyses and \cite{Abbrescia-Speck23} for a comprehensive survey). For the 3D compressible Euler equations, constructing maximal classical development requires understanding CharIVP. CharIVP is crucial for Cauchy horizon formation in shock-forming solutions \cite{Abbrescia-SpeckCauchyHorizon}, as the horizon is a degenerate acoustical null hypersurface. CharIVP also connects to the shock development problem (a characteristic initial boundary value problem) \cite{C19}. Given this understanding, shock reflection and interaction problems \cite{lisibach2021shock,lisibach2022shock,wang2023shock} similarly require CharIVP settings, specifically double-null foliation. There are very few results  regarding CharIVP for the compressible Euler equations. The lack of results has two causes: \textbf{1)} The geometric formulation of 3D compressible Euler with dynamic vorticity and entropy was only recently discovered \cite{Disconzi-Speck19,Luk-Speck20,Speck19}, revealing structures crucial for CharIVP. \textbf{2)} For multiple-speed systems, establishing a functional framework on double null foliations is challenging because null hypersurfaces are characteristic only for part of the system. Specifically, vorticity and entropy propagate transversally to null hypersurfaces, requiring different treatment than the characteristic flow components, while both parts remain coupled.	
	\subsection{A brief account of the related works}\label{SS:RelatedWorks}
	
	We first recall the characteristic initial value problem in the context of general relativity. We consider a given null hypersurface $C$ foliated by $2$-dimensional surfaces $\{S_t\}_{t\in [0,1]}$. We use $\slashed{g}$ and $\chi$ to denote the induced metric on $S_t$ and the second fundamental form of $S_t$ inside $C$ respectively. Let $\widehat{\chi}$ be the traceless part of $\chi$ with respect to $g$. Predicting the characteristic initial value on $C$ for the vacuum Einstein field equations is equivalent to prescribing the conformal class of $\slashed{g}$ on $\{S_t\}$. Moreover, if the geometry of $S_0$ is given, using the Einstein equations it is equivalent to giving $\widehat{\chi}$ on $C$. We refer to Christodoulou's monograph \cite{C07} for a detailed description.
	
	In the case of two intersecting null cones, just as in the last remark of the previous section, once the characteristic initial data are prescribed on two cones, the local well-posedness for the characteristic initial value problem was indeed obtained in 1990 by Rendall, see \cite{Rendall90}. The solution constructed in \cite{Rendall90} is local in the sense that it is in a neighborhood of the intersection of the two null cones.  The solution indeed can be extended to a neighborhood of the initial two null cones; see the work \cite{Luk12} of Luk.
	
	It is also of great interest to study the characteristic initial value problems for the compressible Euler equation; see Section \ref{SS:Motivations} for motivations. In comparison with the Einstein equations, the Euler equations are naturally posed on the Galilean spacetime with usual Cartesian coordinates $(t,x^1,x^2,x^3)$. If a hypersurface $C_0$ is given, we do not have freedom to parametrize in an arbitrary coordinates. For Einstein equations, since the spacetime is part of the solution,  we have extra freedom coming from the diffeomorphism group of the underlying manifold. Therefore, if we prescribe $(\rho, v, s)$ in an arbitrary way on $C_0$, the corresponding acoustical metric is then determined. Thus, this must be a non-trivial constraint to guarantee that $C_0$ is null. To our knowledge, the work \cite{CharivpLisibach} by Lisibach was the only one in this direction. He prescribes the characteristic initial value problem for barotropic flow in the case of spherical symmetry and solves Euler equations locally. In general case,  the only existing work is the preprint \cite{Speck-Yu} of Speck-Yu, under the assumption that the smooth initial data\footnote{The initial data in \cite{Speck-Yu} consists of fluid variables $(v^i,\rho,s)$ on abstract hypersurfaces $\mathring{C_0}\bigcup\mathring{\Cb_0}$ and an embedding $\mathcal{E}$ that maps $\mathring{C_0}\bigcup\mathring{\Cb_0}$ into Cartesian spacetime such that the image $C_0\bigcup\Cb_0$ is a pair of transversally intersecting null hypersurfaces. In particular, in \cite{Speck-Yu}, the two-spheres $S_{u,\ub}$ are not necessarily in $\Sigma_t$.} are given (so are their tangential (w.r.t. the characteristic hypersurfaces) derivatives) on two intersecting characteristic hypersurfaces, they solve the transversal derivatives of the fluid variables on initial characteristic hypersurfaces to form a complete charcteristic initial value problem and they show that there exists a unique solution in a neighborhood of the initial intersecting characteristic hypersurfaces; see Section \ref{SS:CharIVP} for further introductions.

	\subsection{Characteristic initial value problem for 3D compressible Euler equations}\label{SS:CharIVP}
	In this section, we introduce a direct application of Theorem \ref{TH:RoughMain}. In \cite{Speck-Yu}, assuming the fluid data on the initial characteristic hypersurfaces, the authors prove the following local well-posedness result for the characteristic initial value problem: 
	\begin{theorem}[Local well-posedness in a neighborhood of the data null hypersurfaces, \cite{Speck-Yu}]\label{TH:localwellposedness}
		Consider $C^{\infty}$ fluid data for the compressible Euler equations \eqref{eq: Euler in Euler coordinates}
		that are prescribed on the union 
		$C_0 \cup \Cb_0$,
		where $C_0$ and $\Cb_0$ are $C^{\infty}$ characteristic hypersurface portions 
		that intersect transverally in a set $S_{0,0}$ that is diffeomorophic to $\mathbb{S}^2$.
		Assume that the Euler constraint equation\footnote{Euler constraints equation is the Euler equations restricted on the characteristic hypersurfaces. In particular, in this article, the corresponding Euler constraint equation on $C_0$ is \eqref{eq:Euler constraints C}.} is satisfied on
		$C_0$ and $\Cb_0$, and the fluid data satisfy the corner compatibility condition\footnote{The fluid data must satisfy \eqref{eq: Euler in Euler coordinates} on $S_{0,0}$, which results in such compatability conditions for entropy and (angular) velocity.}.
		Let $N \geq 17$ be an integer, 
		and assume that 
		on $C_0([0,\ub_\star]) \cup \Cb_0([0,u_\star])$
		the initial data satisfy the following\footnote{In general, the constructed initial characteristic data suffers loss of derivatives in the transversal direction; see Section \ref{SS:T=1}-\ref{SS:T=k}.} norm bounds:
		\begin{subequations}
			\begin{align}
				\sup\limits_{\ub\in[0,\ub_{\star}]}\norm{\left[\Lb^j(L,\angp)^{k}\right](\rho,v^i,s)}_{L^2(S_{0,\ub})}\leqslant& \Lambda_0,&
				\textnormal{for } & 0\leqslant 2j+k\leqslant N,\\
				\sup\limits_{u\in[0,u_{\star}]}\norm{\left[L^j(\Lb,\angp)^{k}\right](\rho,v^i,s)}_{L^2(S_{u,0})}\leqslant& \Lambda_0,&
				\textnormal{for } & 0\leqslant 2j+k\leqslant N,\\
				\norm{\p^k(\rho,v^i,s)}_{L^\infty(C_{0}\bigcup\Cb_{0})}\leqslant& \Lambda_0,&
				\textnormal{for }& 0\leqslant k\leqslant\frac{N-1}{4},
			\end{align}
		\end{subequations}
		where $\Lambda_0$ is a constant\footnote{$\Lambda_0$ is allowed to be large.}, $L,\Lb$ are principal null vectorfields along $C$ and $\Cb$ respectively, and $\angp$ is the $S_{u,\ub}-$projection of $\p$. We assume the change of variables map\footnote{That is, the map from geometric coordinates to Cartesian coordinates.} $\Upsilon$ is a $C^\infty$-diffeomorphism from $[0,\ub_\star]\times \mathbb{S}^2 \bigcup [0,u_\star]\times \mathbb{S}^2$ onto $C_0\cup\Cb_0$ defined by:
		\begin{align} \label{DE:ChangeofVariableMap}
			\Upsilon(u,\ub,\theta^1,\theta^2)
			& := (t,x^1,x^2,x^3)		
		\end{align}
		
		Then, there exists a number $\delta > 0$ and a corresponding
		future-neighborhood $\mathcal{M}^{u_{\star},\ub_{\star}}_\delta$
		of $C_0([0,\ub_\star]) \cup \Cb_0([0,u_\star])$ 
		such that the following hold on $\mathcal{M}^{u_{\star},\ub_{\star}}_\delta$:
		\begin{itemize}
			\item There exists a unique $C^{\infty}$ solution $(\rho,v^1,v^2,v^3,s)$
			to the compressible Euler equations \eqref{eq: Euler in Euler coordinates}.
			\item There exist unique $C^{\infty}$ solutions $(u,\ub,\theta^1,\theta^2)$ to the initial value problems\footnote{The data of $(u,\ub,\theta^1,\theta^2)$ were given in the construction of the initial characteristic hypersurfaces $C_0\cup\Cb_0$.} subjected to the equations:
			\begin{align} \label{eq: eikonal equations}
				(g^{-1})^{\alpha\beta}\p_\alpha u\p_\beta u=&0,&
				(g^{-1})^{\alpha\beta}\p_\alpha \ub\p_\beta \ub=&0,\\
				\p_{\ub}\theta^1=&0,&
				\p_{\ub}\theta^2=&0.
			\end{align}
			Moreover, the level sets of $u$ and $\ub$ intersect transversally 
			in two-dimensional $g$-spacelike submanifolds 
			$S_{u',\ub'} := \lbrace u = u' \rbrace \cap \lbrace \ub = \ub' \rbrace$,
			all of which are diffeomorphic to $\mathbb{S}^2$.
			\item $\Upsilon$ is a $C^{\infty}$ diffeomorphism from the domain
			$[0,\delta] \times [0,\ub_{\star}] \times \mathbb{S}^2 \bigcup [0,u_{\star}] \times [0,\delta] \times \mathbb{S}^2$
			onto the image set $\mathcal{M}^{u_{\star},\ub_{\star}}_\delta$.
			\item The solution satisfies the energy estimates such that the Sobolev regularity of the fluid variables are propagated as:
			\begin{align}
				\sup\limits_{(u,\ub)\in[0,u_{\star}]\times[0,\ub_{\star}]}\norm{\left[(L,\Lb)^j\angp^{k}\right](\rho,v^i,s)}_{L^2(S_{u,\ub})}\lesssim \Lambda_0\exp\left(\Delta_0(\delta+u_{\star}+\ub_{\star})\right),
			\end{align}
			where $0\leq 2j+k\leq N-4$ and $\Delta_0$ is a large (bootstrap) constant.
		\end{itemize}
	\end{theorem}
	Combining Theorem \ref{TH:RoughMain} and Theorem \ref{TH:localwellposedness}, given hypersurfaces $C_0,\Cb_0$ as in Section \ref{SS:Statement}, we are able to prescribe fluid data such that they are admissible to form a characteristic initial value problem. The solution to this problem (locally) exists and is unique, which suggests that modulo free components of the data ($\vs,s$ in this article; see Section \ref{sec:order 0 2}), the initial characteristic hypersurfaces completely determine the dynamics on the initial hypersurfaces and in the future of them (for at least a short time). In particular, the dynamics of spherical/plane symmetric barotropic compressible Euler flow is uniquely determined by the prescribed initial hypersurfaces.
	\subsection{The spherical symmetric case}\label{SS:intro1D}
	In \cite{CharivpLisibach}, the author considers the characteristic initial value problem in spherical symmetry for barotropic compressible Euler equations, where the characteristic system can be written as an ODE system as follows:\\
	
	\textbf{On $C_u:=\{u=\text{const}\}$:}
	\begin{subequations}
		\begin{align}
			\label{EQ:Riemann1}\frac{\diff\mathcal{R}_{+}}{\diff\ub}=&-\frac{c(\mathcal{R}_{+},\mathcal{R}_{-})}{r}(\mathcal{R}_{+}-\mathcal{R}_{-}),\\
			\frac{\diff r}{\diff\ub}=&\frac{1}{2}(\mathcal{R}_{+}-\mathcal{R}_{-})+c(\mathcal{R}_{+},\mathcal{R}_{-}).
		\end{align}
	\end{subequations}

	\textbf{On $\Cb_{\ub}:=\{\ub=\text{const}\}$:}
	\begin{subequations}
		\begin{align}
			\label{EQ:Riemann2}\frac{\diff\mathcal{R}_{-}}{\diff u}=&-\frac{c(\mathcal{R}_{+},\mathcal{R}_{-})}{r}(\mathcal{R}_{+}-\mathcal{R}_{-}),\\
			\frac{\diff r}{\diff\ub}=&\frac{1}{2}(\mathcal{R}_{+}-\mathcal{R}_{-})-c(\mathcal{R}_{+},\mathcal{R}_{-}),
		\end{align}
	\end{subequations}
	
	where $\mathcal{R}_{+},\mathcal{R}_{-}$ are a pair of Riemann invariants (introduced in \cite{Riemann1860}), $r$ is the radius, $u,\ub$ are the geometric coordinates and $c$ is the speed of sound. In the above system, one has the freedom to prescribe $\mathcal{R}_{-}$ in $C_0 (\{u=0\})$ and $\mathcal{R}_{+}$ on $\Cb_0 (\{\ub=0\})$. 
	
	The characteristic system above provides the following important insights regarding the construction of the initial data for the Euler equations:
	\begin{itemize}
		\item The initial hypersurfaces and the initial fluid data are coupled, so one has to construct the initial hypersurfaces and initial fluid data simultaneously. In other words, the characteristic initial hypersurfaces can/should be viewed as a part of the initial data.
		\item A crucial difference compared to the case of the standard Cauchy problem (in which the data are prescribed on a space-like hypersurface) is that in the characteristic case, \textbf{some components of the initial data are constrained}, that is, determined by the free components by virtue of equations obtained by restricting the PDE to the initial characteristic hypersurfaces.
	\end{itemize}

	Similar to the above model problem, for the 3D compressible Euler equations without symmetry, we have limited freedom\footnote{Similarly, in the context of general relativity, there exists the so-called \emph{free data}, i.e., the 2-tensor $\chih$.} in prescribing part of the initial fluid data. The other components are constrained by the compressible Euler equations and should be solved by PDEs restricted on the initial hypersurfaces. In addition, since characteristic hypersurfaces and fluid variables are deeply coupled, the initial hypersurfaces should be part of the data construction. Moreover, once the data (itself) is prescribed on $C_0$, we also show how to integrate the Euler equations in a specific order to derive all jets of $(\rho, v, s)$ on $C_0$.
	
    \subsection{Acknowledgment}
    SY acknowledges the support of MOE Tier 2 grant A-8000977-00-00.
	
	\section{Acoustical geometry and the precise statement of the result}\label{section_acoustical geometry}
	
	Although in reality solutions to \eqref{eq: Euler in Euler coordinates} may only exist on a portion of the spacetime\footnote{One cannot hope to avoid singularities globally in time: it is known that, even in the irrotational and isentropic case, the compression of sound waves can cause shocks to develop from regular initial data in finite time.}, since all the calculations in this paper is local, we assume a solution $(\rho,v,s)$ exists for all $t\in [0,T]$ and $x\in \mathbb{R}^3$ where $T\in (0,\infty]$. Throughout the paper, we use $(x^0,x^1,x^2,x^3)=(t,x^1,x^2,x^3)$ to denote the Cartesian coordinates on the Galilean spacetime $\mathcal{D}=[0,T]\times \mathbb{R}^3$. We recall that $\Sigma_{t}$ is  the spatial hypersurface $\{(x^0,x^1,x^2,x^3) |x^0=t\}$ and on $[0,T]\times \mathbb{R}^3$. On $\mathcal{D}$, we  have the acoustical metric
	\[ g = - c^2 dt^2 +  \sum_{i=1}^3 (dx^i - v^idt)^2.\]
	This is a four dimensional Lorentzian metric. Its null hypersurfaces represent the characteristic hypersurfaces of the solution $(\rho,v,s)$. In the Cartesian coordinates, the matrix representation of $g$ is given by
	\[(g_{\mu\nu})=\left( {\begin{array}{cccc}
			-c^2+|v|^2 & -v^1 &-v^2 &-v^3 \\
			-v^1 & 1 & 0 & 0\\
			-v^2 & 0 & 1&0\\
			-v^3 & 0 & 0&1
	\end{array} } \right).\]
	The inverse metric can be computed as 
	\[g^{-1} = -\frac{1}{c^2} (\partial_t +\sum_{i=1}^3  v^i \partial_i) \otimes (\partial_t +\sum_{i=1}^3  v^i \partial_i)  +\sum_{i=1}^3 \partial_i \otimes \partial_i,\]
	with matrix form as follows
	\begin{equation}\label{acoustic in inverse matrix}
		(g^{\mu\nu})=\left( {\begin{array}{cccc}
				-\frac{1}{c^2} & -\frac{v^1}{c^2} &-\frac{v^2}{c^2} &-\frac{v^3}{c^2} \\
				-\frac{v^1}{c^2} & 1-\frac{(v^1)^2}{c^2} & -\frac{v^1 v^2}{c^2} & -\frac{v^1 v^3}{c^2}\\
				-\frac{v^2}{c^2} & -\frac{v^2 v^1}{c^2} & 1-\frac{(v^2)^2}{c^2} & -\frac{v^2 v^3}{c^2}\\
				-\frac{v^3}{c^2} & -\frac{v^3 v^1}{c^2} & -\frac{v^3 v^2}{c^2} & 1-\frac{(v^3)^2 }{c^2}
		\end{array} } \right).
	\end{equation}
	We also collect the formulas of the corresponding Christoffel symbols as follows:
	\begin{equation}\label{eq:Christoffel symbols}
		\begin{cases}
			&\Gamma_{00}^0=\frac{1}{2c^2}\left[\partial_0(c^2)+v(-c^2+|v|^2)\right],\\ 
			&\Gamma_{00}^i=\frac{v^i}{2c^2}\left[\partial_0(c^2)+v(-c^2+|v|^2)\right]-\partial_0 v^i-\frac{1}{2}\partial_i(-c^2+|v|^2),\\
			&\Gamma_{i0}^0 =-\frac{1}{2c^2}\partial_i(-c^2+|v|^2)+\frac{v^k}{2c^2}(\partial_iv^k-\partial_kv^i),\\ &\Gamma_{i0}^j =-\frac{v^j}{2c^2}\partial_i(-c^2+|v|^2)+\frac{v^j v^k}{2c^2}(\partial_iv^k-\partial_kv^i)-\frac{1}{2}(\partial_iv^j-\partial_jv^i),\\
			&\Gamma_{ij}^0=\frac{1}{2c^2}(\partial_iv^j+\partial_jv^i),\ \ \Gamma_{ij}^k=\frac{v^k}{2c^2}(\partial_iv^j+\partial_jv^i).
		\end{cases}
	\end{equation}

	\subsection{The acoustical coordinate system}
	
	We refer to \cite{C07} and \cite{C-Miao14} for details of the construction of the acoustical coordinates. A acoustical (local) coordinate system on $\mathcal{D}=[0,T]\times \mathbb{R}^3$ consists of four smooth functions $t$, $u$ and $\vartheta^A$ with $A=1,2$. The temporal function $t$ is defined as  $t=x^0$.

	We assume that there exists a smooth function $u_0\in C^\infty(\Sigma_0)$ so that $S_{0,0}=C_0\cap \Sigma_0=u_0^{-1}(0)$ and $\nabla u_0\big|_{S_{0,0}}\neq 0$. We also assume that $u_0>0$ corresponds to the interior of $S_{0,0}$ inside $\Sigma_0$. This is the region where $(\rho, v,s)$ is already given. Moreover, the level sets of $u_0$ provide a smooth foliation of $\Sigma_0$ around $S_{0,0}$.
	
	To construct the acoustical function $u$, first of all, we require the acoustical function $u$ is equal to $u_0$ on $\Sigma_{0}$. We then define $C_u$ to be the null hypersurfaces consisting of null geodesics emanating from each level set of $u$ on $\Sigma_0$. Then, we require $C_u$ to be the level sets of $u$ and this defines $u$ on $\mathcal{D}$. We also define $S_{t,u} = \Sigma_t \cap C_u$. Therefore, we have $\Sigma_t =\displaystyle \bigcup_{u}S_{t,u}$. 
	\begin{center}
		\includegraphics[width=4in]{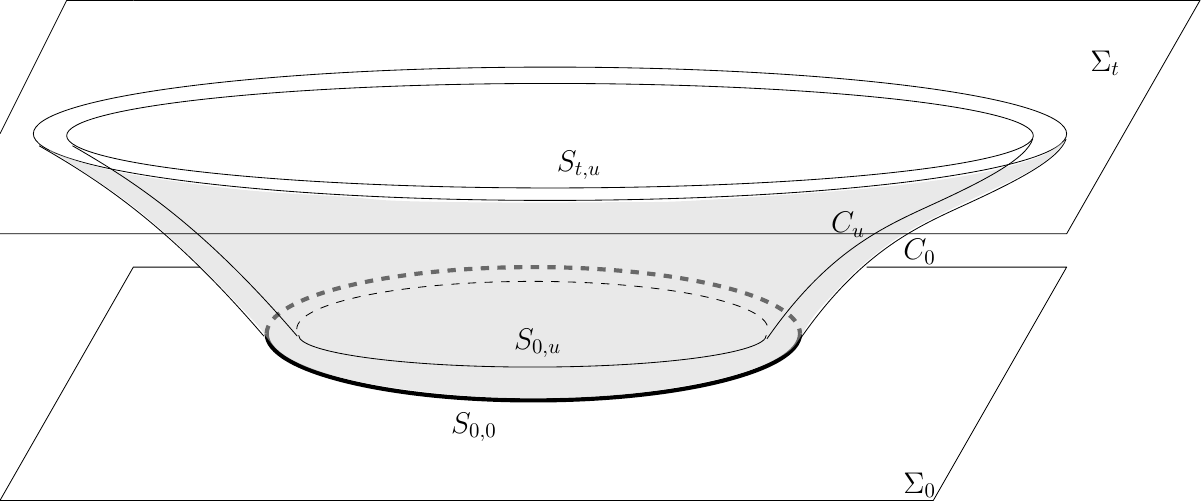}
	\end{center}
	We choose {\color{black}the future-pointed vector field $L$} to be the generators of the null geodesics on $C_u$ in such a way that $L(t) = 1$. The inverse density function $\mu$ measures the temporal density of the foliations $\{C_u\}_{u\geqslant 0}$ and it is defined as
	\[\mu^{-1} = -g(D t,D u),\]
	where $D$ is the Levi-Civita connection with respect to the acoustical metric $g$.
	
	We then consider the normal vector field $T$ of the foliation $\displaystyle \bigcup_{u}S_{t,u}$ on $\Sigma_t$. It is uniquely defined by the following three conditions:
	\[(1) \ T \  \text{is tangent to} \ \Sigma_t; \ \ (2) \ {\color{black}T \ \text{is $g$-perpendicular to}} \ S_{t,u}; \ \ (3) \ Tu = 1.\]
	

	To define the angular coordinate functions $\vartheta^1$ and $\vartheta^2$,  for all $u$, we can first fix a local coordinate system $(\vartheta^A)_{A=1,2}$ on  $S_{0,0}$. The next step is to define $\vartheta$ on $\Sigma_0$ by extending $(\vartheta^A)_{A=1,2}$ through the following equation on $\Sigma_0$:
	\[
	T( \vartheta^A)=0,\ \ \ \vartheta^A \big|_{S_{0,0}}= \vartheta^A .
	\]
	Since the $1$-parameter diffeomorphism group generated by $L$ maps $S_{0,u}$ to $S_{t,u}$, we then take $(\vartheta^A)_{A=1,2}$ as local coordinate system on $S_{t,u}$ by the above $1$-parameter group, i.e.,  we use the following ODE system to extend $(\vartheta^A)_{A=1,2}$:
	\[L( \vartheta^A)=0, \ \ A=1,2.\]
	Therefore, we have constructed an acoustical coordinate system $(t,u,\vartheta^1,\vartheta^2)$ on $\mathcal{D}$. 
	
	We also denote $X_A= \frac{\partial}{\partial \vartheta^A}$. In the acoustical coordinates $(t,u,\vartheta)$, we have
	\begin{equation}\label{eq: L T in terms of coordinates}
		L = \frac{\partial}{\partial t}, \quad T = \frac{\partial}{\partial u} -\sum_{A=1,2} \Xi^A \frac{\partial}{\partial \vartheta^A}, \quad \frac{\partial}{\partial \vartheta^A}=X_A,
	\end{equation}
	where $\Xi^A (A=1,2)$ are smooth functions. The acoustical metric then can be expressed as
	\[
	g=-\mu (du\otimes dt+dt\otimes du)+\kappa^2 du\otimes du +\slashed{g}_{AB}(\Xi^A du+d\vartheta^A)\otimes (\Xi^B du+ d\vartheta^B),
	\]
	where $\slashed{g}_{AB}=g(X_A,X_B)$ and the repeated capital latin indices $A$ and $B$ are understood as summations over $\{1,2\}$.
	
	Let $B=\partial_t+\sum_{i=1}^3 v^i\partial_i$ be the traditional material derivative vector field. We then have $B(t)=1$ and \[ g(B,B)=-c^2, \ \ g(B, \partial_i)=0, \ i=1,2,3.\] 
	This shows that $c^{-1}B$ is the future-pointed unit normal to $\Sigma_t$.
	
	Let $\kappa^2=g(T,T)$, we can compute that $\mu = c \kappa$. We also define the unit vector $\widehat{T} = \kappa^{-1}T$. The null vector field $L$ can be represented as 
	\[L = \frac{\partial}{\partial t} +v-c\widehat{T}=\partial_t-\sum_{i=1}^3(c\Th^i-v^i)\frac{\partial}{\partial x^i},\]
	where $\widehat{T}=\sum_{i=1}^3 \Th^i\frac{\partial}{\partial x^i}$ is the expression of $\Th$ on $\Sigma_t$ in the Cartesian coordinates.

	\subsection{The acoustical structure equations}
	
	We first introduce three different fundamental forms. For the embedded spacelike hypersurface $\Sigma_t\hookrightarrow \mathcal{D}$, since $c^{-1}B$ is the unit normal, we define
	\[k(X,Y)=g\left(D_Y \big(c^{-1}B\big),X\right) \ \ \Leftrightarrow \ \ ck(X,Y)=g\left(D_Y B ,X\right),\]
	where $X,Y$ are tangent vectors to $\Sigma_t$.
	In the Cartesian coordinates, we have
	\[ ck_{ij}=\frac{1}{2}(\partial_iv^j+\partial_jv^i).\]
	For the null hypersurfaces $C_u$, we define
	\[\chi(X,Y):=g(\nabla_X L,Y),\]
	where $X,Y$ are tangent vectors to $S_{t,u}$.
	Similarly, for the embedding $S_{t,u}\hookrightarrow \Sigma_t$, using its unit normal $\Th$, we define
	\[\theta(X,Y)=g\left(\nabla_Y \Th, X\right) \ \ \Leftrightarrow \ \ \kappa \theta(X,Y)=g\left(\nabla_Y T ,X\right),\]
	where $X,Y$ are tangent vectors to $S_{t,u}$. The $\nabla$ is the Levi-Civita connection to the standard Euclidean metric  on $\Sigma_t$. We remark that $\nabla=D\big|_{\Sigma_t}$ since $g\big|_{\Sigma_t}$is the standard Euclidean metric. By $L=B-c\Th=c(c^{-1}B-\Th)$, for all $X,Y$ tangent vectors to $S_{t,u}$, we have
	\[\chi(X,Y)=c\big(k(X,Y)-\theta(X,Y)\big).\]
	
	We also define torsion forms
	\[\zeta(X):=g(D_X L, T), \ \ \eta(X):=-g(D_X T,L),\]
	where  $X$ is a tangent vector field to $S_{t,u}$. We can check that
	\[\eta = \zeta+\slashed{d}(\mu),\]
	where $\slashed{d}$ is the differential on $S_{t,u}$. The torsion $\zeta$ can also be expressed as
	\begin{equation}\label{eq: zeta in X and T} 
		\zeta(X)=\kappa\big(ck(X, \Th)-X(c)\big),
	\end{equation}
	where  $X$ is a tangent vector field to $S_{t,u}$.

	By the above notations, we can express the Levi-Civita connection in the frame $(L,T,X^1,X^2)$ as follows:
	\begin{equation}\label{eq:structure equations}
		\begin{cases}
			D_L L&= \mu^{-1}L(\mu)L,\\
			D_TL&= \eta^A X_A-c^{-1}L(\kappa)L,\\
			D_LT&= -\zeta^A X_A-c^{-1}L(\kappa)L,\\
			D_{X_A} L &= -\mu^{-1}\zeta_A L+\chi_A{}^{B}X_B,\\
			D_{X_A}T&=c^{-1}\kappa \varepsilon_A L+\mu^{-1}\eta_A T+\kappa\theta_{A}{}^{B}X_B,\\
			D_{T}T&=c^{-2}\kappa(Tc+L\kappa)L+\big[c^{-1}(Tc+L\kappa)+\kappa^{-1}T\kappa\big]T-\kappa\slashed{g}^{AB}X_B(\kappa)X_A,\\
			D_{X_A}X_B&=\slashed{D}_{X_A}X_B+c^{-1}\slashed{k}_{AB}L+\mu^{-1}\chi_{AB}T.
		\end{cases}
	\end{equation}
	We refer to \cite[Section 3]{C-Miao14} for detailed computations.
	
	From \eqref{eq:structure equations}, it is straightforward to derive the following commutation formulas:
	\begin{equation}\label{eq:commutator}
		[L,T]=-\left(\zeta^A+\eta^A\right)X_A, \ [L,X_A]=0, \ 
	\end{equation}
	Since $T=\kappa \Th$, we also have
	\begin{equation}\label{eq:commutator 2}
		[L,\Th]=-\frac{1}{\kappa}\left(\zeta^A+\eta^A\right)X_A-\frac{1}{\kappa}L(\kappa)\Th.
	\end{equation}
	
	
	The most important structure equation in the paper is the transport equation for the inverse density $\kappa$ along $L$: 
	\begin{equation}\label{eq: structure equation L kappa}
		L\kappa=-Tc+\sum_{j=1}^3\Th^j T(v^j). 
	\end{equation}
	where $\Th^i$ are functions. We also recall that $\widehat{T}=\sum_{i=1}^3 \Th^i\frac{\partial}{\partial x^i}$ is the expression of $\Th$ on $\Sigma_t$ in the Cartesian coordinates. For $\Th^i$'s, its derivatives can be expressed as
	\begin{equation}\label{eq: L Th i}
		L(\Th^i)\partial_i=\slashed{g}^{AB}\left[-\Th^jX_A(v^j)+X_A(c)\right]X_B,
	\end{equation}
	\begin{equation}\label{eq: T Th i}
		T(\Th^i)\partial_i=-\slashed{g}^{AB}X_B (\kappa)X_A,
	\end{equation}
	and
	\begin{equation}\label{eq: XA Th i}
		X_A(\Th^i)\partial_i=\slashed{g}^{CB}\theta_{BA} X_C.
	\end{equation}

	By the formula \eqref{eq:structure equations} of the covariant derivatives and the fact that $\chi =c (\slashed{k} -\theta)$, we can decompose the wave operator $\Box_g$ in the null frame $\{E_1,E_2,E_3,E_4\}=\{X_1,X_2,T,L\}$. Indeed, for a given smooth function $f$, we can compute that
	\begin{equation}\label{eq: wave in L T}
		\begin{split}
			\Box_g f&=\sum_{\mu,\nu=1}^4 g^{E_{\mu}E_{\nu}}\left(E_\mu\big(E_\nu(f)\big)-\big(D_{E_\mu}{E_\nu}\big)(f)\right)\\
			&=-\frac{2}{\mu} LT(f)+\slashed{\triangle}_{\slashed{g}}(f)-c^{-2}L^2(f)\\
			&\ \ \ +c^{-3}L(c)L(f)-(\mu c)^{-1}L(\kappa)L(f)-c^{-1}\tr\slashed{k}\cdot L(f)-(\tr\slashed{k}-\tr\theta)\Th(f)-\frac{2}{\mu}\zeta^A X_A(f).
		\end{split}
	\end{equation}
	
	\subsection{Wave equation for the density function}
	
	We derive a wave equations with respect to the acoustical metric for the density function $\rho$ in this section. We first rewrite the Euler equations in a geometric form. In terms of $B=\partial_t+v\cdot \nabla$, we can write the Euler equations as follows:
	\begin{equation}\label{eq:Euler equations rho}
		\begin{cases}
			&B(\rho)=-\rho\cdot {\rm div}(v),\\
			&B(v^i)=-\rho^{-1} \nabla^i p,\\
			&B(s)=0.
		\end{cases}
	\end{equation}
	For an arbitrary vector field $X$ on $\mathcal{D}$,  we can write it uniquely in the following form
	\[X=-\frac{1}{c^2}g(X,B)B+X^{\parallel},\]
	where $X^\parallel$ is tangential to $\Sigma_t$.
	For arbitrary smooth function $\phi$, in the orthonormal frame $\{e_0=B,e_1=\partial_1,e_2=\partial_2,e_3=\partial_3\}$,  we compute
	\begin{align*}
		\Box_g \phi &=g^{\mu\nu}\big[e_\mu e_\nu (\phi)-(\nabla_{e_\mu}e_\nu)(\phi)\big]=g^{BB}\big[B^2(\phi)-(\nabla_{B}B)(\phi)\big]+g^{ij}\big[e_ie_j (\phi)-(\nabla_{e_i}e_j)(\phi)\big]\\
		&=g^{BB}\big[B^2(\phi)-(\nabla_{B}B)(\phi)\big]+g^{ij}\big[e_i e_j (\phi)-(\nabla_{e_i}e_j)^\parallel(\phi)+\frac{1}{c^2}g(\nabla_{e_i}e_j,B)B(\phi)\big]\\
		&=-\frac{1}{c^2}B^2(\phi)+\frac{1}{c^2}(\nabla_{B}B)(\phi)+\triangle \phi-\frac{g^{ij}k_{ij}}{c}B(\phi)\\
		&=-\frac{1}{c^2}B^2(\phi)+\frac{1}{c^2}(\nabla_{B}B)(\phi)+\triangle \phi-\frac{{\rm div}(v)}{c^2}B(\phi).\end{align*}
	Since $\nabla_B B$ can be expressed as
	\begin{align*}
		\nabla_B B=\frac{1}{c}B(c)B+c\sum_{i=1}^3\partial_i(c)\partial_i,
	\end{align*}
	we obtain that
	\begin{equation}\label{eq:formula for box g phi}
		\Box_g \phi =-\frac{1}{c^2}B^2(\phi)+\frac{1}{c^3}B(c)B(\phi)+\frac{1}{c}\partial_i(c)\partial_i(\phi)+\triangle \phi-\frac{{\rm div}(v)}{c^2}B(\phi).
	\end{equation}

	By \eqref{eq:Euler equations rho}, we can derive the wave equation for $\rho$ as follows:
	\begin{align*}
		\Box_g(\rho)&=-\frac{1}{c^2}B^2(\rho)+\frac{1}{c^3}B(c)B(\rho)+\frac{1}{c}\nabla (c)\cdot \nabla(\rho)+\triangle \rho-\frac{{\rm div}(v)}{c^2}B(\rho)\\
		&=\frac{1}{c^2}B\left(\rho {\rm div}(v)\right)+\frac{1}{c^3}B(c)B(\rho)+\frac{1}{c}\nabla (c)\cdot \nabla(\rho)+\triangle \rho-\frac{{\rm div}(v)}{c^2}B(\rho)\\
		&=\frac{\rho}{c^2}   B \big({\rm div}(v)\big)+\frac{1}{c^3}B(c)B(\rho)+\frac{1}{c}\nabla (c)\cdot \nabla(\rho)+\triangle \rho\\
		&=\frac{\rho}{c^2} \left(\frac{\nabla \rho \cdot \nabla p}{\rho^2}-\frac{\triangle p}{\rho}-\nabla_i v^j \nabla_j v^i\right)+\frac{1}{c^3}B(c)B(\rho)+\frac{1}{c}\nabla (c)\cdot \nabla(\rho)+\triangle \rho.
	\end{align*}
	We notice that 
	\[\nabla^i p=c^2\nabla^i \rho+\frac{\partial p}{\partial s} \nabla^i s,\] 
	and
	\[\triangle p=c^2\triangle \rho+\frac{\partial p}{\partial s} \triangle s+\frac{\partial  c^2}{\partial \rho}|\nabla \rho|^2 +2\frac{\partial  c^2}{\partial s} \nabla \rho\cdot \nabla s+\frac{\partial^2 p}{\partial s^2} |\nabla s|^2.\]
	Therefore,
	\begin{align*}
		\Box_g(\rho)=&\frac{1}{c^2} \Big\{\frac{ c^2}{\rho}|\nabla \rho|^2+\frac{1}{\rho}\frac{\partial p}{\partial s} \nabla \rho \cdot \nabla s- \big[c^2\triangle \rho+\frac{\partial p}{\partial s} \triangle s+\frac{\partial  c^2}{\partial \rho}|\nabla \rho|^2 +2\frac{\partial  c^2}{\partial s} \nabla \rho\cdot \nabla s+\frac{\partial^2 p}{\partial s^2} |\nabla s|^2\big] \\
		&-\rho \nabla_i v^j \nabla_j v^i\Big\}  +\frac{1}{c^3}B(c)B(\rho)+\frac{1}{c}\nabla (c)\cdot \nabla(\rho)+\triangle \rho\\
		=&\frac{1}{c^2} \Big\{\frac{ c^2}{\rho}|\nabla \rho|^2+\frac{1}{\rho}\frac{\partial p}{\partial s} \nabla \rho \cdot \nabla s- \big[\frac{\partial p}{\partial s} \triangle s+2c\frac{\partial  c}{\partial \rho}|\nabla \rho|^2 +4c\frac{\partial  c }{\partial s} \nabla \rho\cdot \nabla s+\frac{\partial^2 p}{\partial s^2} |\nabla s|^2\big] \\
		&-\rho \nabla_i v^j \nabla_j v^i\Big\}
		+\frac{1}{c^3}B(c)B(\rho)+\frac{1}{c}\frac{\partial c}{\partial \rho}|\nabla(\rho)|^2 +\frac{1}{c}\frac{\partial c}{\partial s}\nabla (\rho)\cdot \nabla(s) \\
		=&\frac{1}{c^2} \Big\{\frac{ c^2}{\rho}|\nabla \rho|^2+\frac{1}{\rho}\frac{\partial p}{\partial s} \nabla \rho \cdot \nabla s- \big[\frac{\partial p}{\partial s} \triangle s+\frac{\partial  c^2 }{\partial s} \nabla \rho\cdot \nabla s+\frac{\partial^2 p}{\partial s^2} |\nabla s|^2\big] -\rho \nabla_i v^j \nabla_j v^i\Big\}\\
		&+\frac{1}{c^3}B(c)B(\rho)-\underbrace{\left(\frac{1}{c}\frac{\partial c}{\partial \rho}|\nabla(\rho)|^2 +\frac{1}{c}\frac{\partial c}{\partial s}\nabla (\rho)\cdot \nabla(s) \right)}_{\frac{1}{c}\nabla (c)\cdot \nabla(\rho)}\\
		=&\frac{1}{c^2} \left(\frac{ c^2}{\rho}|\nabla \rho|^2 -\rho \nabla_i v^j \nabla_j v^i\right)+\frac{1}{c^2} \Big\{ \left(\frac{1}{\rho}\frac{\partial p}{\partial s} -\frac{\partial^2 p}{\partial s \partial \rho}\right)\nabla \rho \cdot \nabla s- \big[\frac{\partial p}{\partial s} \triangle s+\frac{\partial^2 p}{\partial s^2} |\nabla s|^2\big] \Big\}\\
		&-\frac{g(Dc,D\rho)}{c}.
	\end{align*}
	Hence,
	\begin{equation*}
		\Box_g(\rho)= \frac{ 1}{\rho}|\nabla \rho|^2 -\frac{\rho}{c^2} \nabla_i v^j \nabla_j v^i -\frac{g(Dc,D\rho)}{c}+\frac{1}{c^2} \left\{ \left(\frac{1}{\rho}\frac{\partial p}{\partial s} -\frac{\partial^2 p}{\partial s \partial \rho}\right)\nabla \rho \cdot \nabla s- \big[\frac{\partial p}{\partial s} \triangle s+\frac{\partial^2 p}{\partial s^2} |\nabla s|^2\big] \right\}.
	\end{equation*}
	In view of
	\begin{align*}
		|\nabla \rho|^2=g(D\rho,D\rho)+\frac{B(\rho)^2}{c^2}=g(D\rho,D\rho)+\frac{\rho^2}{c^2}{\rm div}(v)^2,
	\end{align*}
	we conclude that 
	\begin{equation}\label{eq: box rho}
		\begin{split}
			\Box_g(\rho)&= \frac{\rho}{c^2} \big({\rm div}(v)^2-\nabla_i v^j \nabla_j v^i\big) +\frac{g(D\rho,D\rho)}{\rho}-\frac{g(Dc,D\rho)}{c}\\
			& \ \ +\frac{1}{c^2} \left\{ \left(\frac{1}{\rho}\frac{\partial p}{\partial s} -\frac{\partial^2 p}{\partial s \partial \rho}\right)\nabla \rho \cdot \nabla s- \big[\frac{\partial p}{\partial s} \triangle s+\frac{\partial^2 p}{\partial s^2} |\nabla s|^2\big] \right\}
		\end{split}
	\end{equation}

	\subsection{Transport equations for Riemann invariants along characteristic hypersurfaces}

	On each $\Sigma_t$, we consider the orthogonal decomposition of $v$ as 
	\[v=v^{\Th}\cdot \Th+\vs,\]
	where $\vs$ is tangent to $S_{t,u}$. In view of the notations firstly introduced in Riemann's original paper \cite{Riemann1860}, we define the \emph{Riemann invariants} with respect to the null hypersurface $C_u$ as follows:
	\begin{equation}\label{def: Riemann invariants}
		\begin{cases}
			\wb=\frac{1}{2}\left(\int \frac{c(\rho,s)}{\rho} d\rho- v^{\Th}\right), \\ 
			w=\frac{1}{2}\left(\int \frac{c(\rho,s)}{\rho} d\rho+v^{\Th}\right),\\
			\ws=\slashed{v}.
		\end{cases}
	\end{equation}
	For the sake of simplicity, we introduce two auxiliary functions
	\[\Phi_0(\rho,s)=\int \frac{c(\rho,s)}{\rho} d\rho, \ \ \Psi_0(\rho,s)=\int \frac{1}{\rho}\frac{\partial c}{\partial s} d\rho-\frac{1}{c\rho}\frac{\partial p}{\partial s}.\]
	It is straightforward to check that the Euler equations \eqref{eq:Euler equations rho} can be written as
	\begin{equation}\label{eq:Euler equations prepare to derive Riemann} 
		\begin{cases}
			B(\Phi_0) &= -c\cdot {\rm div} (v),
			\\
			B(v^i) &= -c\cdot \nabla_i \Phi_0+c \cdot \Psi_0 \cdot \nabla_i s,\\
			B(s)&=0.
		\end{cases}
	\end{equation} 
	The purpose of the current section is to derive transport equations for $w$ and $\wb$ along $C_u$. We first transform \eqref{eq:Euler equations prepare to derive Riemann} in a covariant form so that it suits the tensorial computations.
	
	In view of the definition $\theta=g(D_X\Th,Y)$, we have
	\begin{equation}\label{eq: formula for div v}
		{\rm div}(v)=\Th(v_\Th)+\slashed{\rm div}(\slashed{v})-g\left( \slashed{v}, \nabla_\Th \Th\right)+v_\Th \cdot {\rm tr} _{\slashed{g}} \theta.
	\end{equation}
	In view of of the definition $k(X,Y)=g(\nabla_X (c^{-1}B),Y)$, we have
	\begin{align*}
		(\partial_t + v \cdot \nabla)v 
		&=-D_v \partial_0+ D_{B} v-c^{-1}k(v,v)B.
	\end{align*}
	By \eqref{eq:Christoffel symbols}, we compute
	\begin{align*}
		D_v \partial_0&=v^i\big(\Gamma_{i0}^0\partial_0+\Gamma_{i0}^j\partial_j\big)
		=-\frac{1}{2c^2}v(-c^2+|v|^2)B-\frac{1}{2}v^i(\partial_iv^j-\partial_j v^i)\partial_j\\
		&=-\frac{1}{2c^2}v(-c^2+|v|^2)B-\frac{1}{2}\nabla_v v+\frac{1}{4}\nabla(|v|^2)
	\end{align*}
	In view of $ck_{ij}=\frac{1}{2}(\partial_iv^j+\partial_jv^i)$, we have
	\begin{align*}
		\frac{1}{2c^2}v(-c^2+|v|^2)B -c^{-1}k(v,v)B=-\frac{1}{c}v(c)B.
	\end{align*}
	Therefore, 
	\begin{align*}
		(\partial_t + v \cdot \nabla)v&= D_{B} v-c^{-1}v(c)B+\frac{1}{2}v^i(\partial_iv^j-\partial_j v^i)\partial_j\\
		&= D_{B} v-c^{-1}v(c)B+\frac{1}{2}\nabla_v v-\frac{1}{4}\nabla(|v|^2).
	\end{align*}
	
	By this equation together with \eqref{eq: formula for div v},  the equation \eqref{eq:Euler equations prepare to derive Riemann} then can be rewritten as
	\[
	\begin{cases}
		&B (\Phi_0) =-c \left[\Th(v_\Th)+\slashed{\rm div}(\slashed{v})-g\left( \slashed{v}, \nabla_\Th \Th\right)+v_\Th \cdot  {\rm tr} _{\slashed{g}} \theta\right],\\
		&D_{B} v- c^{-1}v(c) \cdot B+\frac{1}{2}\nabla_v v-\frac{1}{4}\nabla(|v|^2)=-c \left[\Th \left(\Phi_0\right) \Th +g^{AB}X_B \left(\Phi_0\right) X_A\right] \\
		&\qquad\qquad\qquad\qquad\qquad\qquad\qquad\qquad\qquad+c \Psi_0 \left[\Th \left(s\right) \Th +g^{AB}X_B \left(s\right) X_A\right] .
	\end{cases}
	\]
	By $L=B-c\Th$ and $v=v^{\Th}\Th+\slashed{v}$, we obtain that 
	\begin{equation}\label{eq:Euler transformed 1}
		\begin{cases}
			&L\left(\Phi_0\right) = -c \left[\Th\left(\Phi_0+v_\Th\right)+\slashed{\rm div}(\slashed{v})-g\left( \slashed{v}, \nabla_\Th \Th\right)+v_\Th \cdot {\rm tr} _{\slashed{g}} \theta\right],
			\\
			& D_{L} v+cD_{\Th} v-  c^{-1}v(c) \cdot L -v(c) \cdot\Th+\frac{1}{2}\nabla_v v-\frac{1}{4}\nabla(|v|^2)\\
			&\qquad \quad =-c \left[\Th \left(\Phi_0\right) \Th +g^{AB}X_B \left(\Phi_0\right) X_A\right] +c \Psi_0 \left[\Th \left(s\right) \Th +g^{AB}X_B \left(s\right) X_A\right]  .
		\end{cases}
	\end{equation}
	To proceed, we then further decompose $D_{L} v$ and $D_{\Th} v$ using the frame $(L,T,X^1,X^2)$. Since $D_LT= -\zeta^A X_A-c^{-1}L(\kappa)L$, we have
	\begin{align*}
		D_{L} v
		&=L(v^{\Th}) \cdot \Th+ D_{L}\slashed{v} -c^{-1}\kappa^{-1}L(\kappa)\cdot v^{\Th} \cdot L-\kappa^{-1}L(\kappa)\cdot v^{\Th} \cdot \Th-\kappa^{-1}\zeta^A \cdot v^{\Th}  \cdot X _A  .
	\end{align*}
	In view of the following formula:
	\[D_{T}T=c^{-2}\kappa(Tc+L\kappa)L+\big[c^{-1}(Tc+L\kappa)+\kappa^{-1}T\kappa\big]T-\kappa\slashed{g}^{AB}X_B(\kappa)X_A,
	\]
	we obtain that
	\begin{equation}\label{eq:structure D Th Dh}D_{\Th}\Th=c^{-2}\big(\Th(c)+\kappa^{-1}L(\kappa)\big)\big(L+c\Th\big)-\kappa^{-1}\slashed{g}^{AB}X_B(\kappa)X_A.
	\end{equation}
	We can use this formula to compute
	\begin{align*}
		D_{\Th} v
		=& D_{\Th}\slashed{v}
		+c^{-2}\big(\Th(c)+\kappa^{-1}L(\kappa)\big)\cdot  v^{\Th}\cdot L+\big[\Th(v^{\Th}) +c^{-1}\big(\Th(c)+\kappa^{-1} L(\kappa)\big)\cdot  v^{\Th}\big] \Th\\
		&-\kappa^{-1}\slashed{g}^{AB}X_B(\kappa)\cdot  v^{\Th}\cdot X_A.
	\end{align*}
	Finally, by these formulas, we can change \eqref{eq:Euler transformed 1} into the following form:
	\begin{equation}\label{eq:Euler transformed final}
		\begin{cases}
			L\left(\Phi_0\right) &= -c \left[\Th\left(\Phi_0+v_\Th\right)+\slashed{\rm div}(\slashed{v})-g\left( \slashed{v}, \nabla_\Th \Th\right)+v_\Th \cdot {\rm tr} _{\slashed{g}} \theta\right],
			\\
			D_{L}\slashed{v}+&\big[L(v^{\Th}) +c\Th(v^{\Th}) -\slashed{v}(c)\big] \Th +cD_{\Th}\slashed{v}+\frac{1}{2}\nabla_v v-\frac{1}{4}\nabla(|v|^2)\\
			&= c^{-1}\slashed{v}(c) \cdot L+\kappa^{-1}\left(\zeta^A+c \slashed{g}^{AB}X_B(\kappa)\right) \cdot v^{\Th}  \cdot X _A \\
			& \ \ \ -c \left[\Th \left(\Phi_0\right) \Th +g^{AB}X_B \left(\Phi_0\right) X_A\right] +c \Psi_0 \left[\Th \left(s\right) \Th +g^{AB}X_B \left(s\right) X_A\right]  .
		\end{cases}
	\end{equation}
	We now make use of the second equation. First of all, we take the inner product of the second equation of \eqref{eq:Euler transformed final} with $L$. We first compute that
	\begin{align*}
		g\big( \frac{1}{2}\nabla_v v-\frac{1}{4}\nabla(|v|^2),L\big)&=g\big( \frac{1}{2}\nabla_v v-\frac{1}{4}\nabla(|v|^2),-c\Th\big)\\
		&=- \frac{c}{2}\left[v^{\Th}g(\nabla_{\Th}\slashed{v},\Th)+\slashed{v}(v^{\Th})+g(\nabla_{\slashed{v}} \slashed{v},\Th)-\frac{1}{2}\Th(|\slashed{v}|^2)\right].
	\end{align*}
	In view of $D_TL= \eta^A X_A-c^{-1}L(\kappa)L$ and $g(L,\Th)=-c$, we derive
	\begin{equation}\label{eq: Euler inner product with L}
		\begin{split}
			L(v^{\Th})&+c\Th(v^{\Th})+\kappa^{-1}\eta_A\slashed{v}^A-\slashed{v}(c)+c \Th \left(\Phi_0\right) \\
			&+ \frac{1}{2}\left[v^{\Th}g(\nabla_{\Th}\slashed{v},\Th)+\slashed{v}(v^{\Th})+g(\nabla_{\slashed{v}} \slashed{v},\Th)-\frac{1}{2}\Th(|\slashed{v}|^2)\right]=c \Psi_0 \cdot\Th \left(s\right).
		\end{split}
	\end{equation}
	Secondly, we take the inner product of the second equation of \eqref{eq:Euler transformed final} with $\Th$. In view of \eqref{eq:structure equations} and \eqref{eq:structure D Th Dh}, we derive
	\begin{equation}\label{eq: Euler inner product with Th}
		\begin{split}
			L(v^{\Th})&+c\Th(v^{\Th})+\kappa^{-1}\zeta_A\slashed{v}^A+c\kappa^{-1}\slashed{v}(\kappa)+c \Th \left(\Phi_0\right)\\
			&+\frac{1}{2}\left[v^{\Th}g(\nabla_{\Th}\slashed{v},\Th)+\slashed{v}(v^{\Th})+g(\nabla_{\slashed{v}} \slashed{v},\Th)-\frac{1}{2}\Th(|\slashed{v}|^2)\right]=c \Psi_0 \Th \left(s\right).
		\end{split}
	\end{equation}
	\begin{remark}
		Since  $\eta=\zeta+\slashed{d} \mu$, \eqref{eq: Euler inner product with Th} is equivalent to \eqref{eq: Euler inner product with L}.In view of $\zeta=\kappa (c\varepsilon-\slashed{d}(c))$ where $\kappa\varepsilon(X)=k(X,T)$, we have
		\begin{align*}
			\kappa^{-1}\zeta(\slashed{v})&=ck(\slashed{v}, \Th)- \slashed{v}(c)=\frac{1}{2}g(\nabla_\slashed{v} v,\Th)+\frac{1}{2}g(\nabla_{\Th} v,\slashed{v} )-\slashed{v}(c)\\
			&=-\frac{1}{2}v^{\Th}g(\nabla_{\Th}\slashed{v},\Th)+\frac{1}{2}\slashed{v}(v^{\Th})+\frac{1}{2}g(\nabla_{\slashed{v}} \slashed{v},\Th)+\frac{1}{4}\Th(|\slashed{v}|^2)-\slashed{v}(c)
		\end{align*}
		Since $g(\nabla_{\slashed{v}} \slashed{v},\Th)=-\theta(\slashed{v},\slashed{v})$, \eqref{eq: Euler inner product with Th} and \eqref{eq: Euler inner product with L} are equivalent to
		\begin{equation}\label{eq: Euler inner product with Th final}
			L(v^{\Th})+c\Th(v^{\Th})+c\kappa^{-1}\slashed{v}(\kappa)+c \Th \left(\Phi_0\right) +\slashed{v}(v^{\Th}-c)-\theta(\slashed{v},\slashed{v})=c \Psi_0 \Th \left(s\right).
		\end{equation}
	\end{remark}

	To summarize, we can derive from the second equation of \eqref{eq:Euler transformed final} that 
	\begin{equation}\label{eq:Euler transformed final 1}
		L(v^{\Th})+c \Th \left(v^{\Th}+\Phi_0\right)+c\kappa^{-1}\slashed{v}(\kappa)+\slashed{v}(v^{\Th}-c)-\theta(\slashed{v},\slashed{v})=c \Psi_1 \Th \left(s\right),
	\end{equation}
	
	In view of \eqref{eq:structure D Th Dh}, we have $g\left( \slashed{v}, \nabla_\Th \Th\right)=-\kappa^{-1}\slashed{v}(\kappa)$. Therefore,  the first equation in \eqref{eq:Euler transformed final} can be written as
	\begin{equation}\label{eq:Euler transformed final 3}
		L\left(F_0\right) +c  \Th\left(v_\Th+F_0\right)+c\slashed{\rm div}(\slashed{v})+c\kappa^{-1}\slashed{v}(\kappa)+cv_\Th \cdot {\rm tr} _{\slashed{g}} \theta=0.
	\end{equation} 
	In view of the definition of Riemann invariants, i.e., \eqref{def: Riemann invariants}, the equations \eqref{eq:Euler transformed final 1} and \eqref{eq:Euler transformed final 3} can be transformed into the following form:
	\begin{subequations}\label{eq:Euler transformed final 4} 
		\begin{align}
			\label{eq:Euler constraints C}L(\wb)=&-\frac{1}{2} \left[c\slashed{\rm div}(\slashed{v})+c{\rm tr} _{\slashed{g}} \theta \cdot v^{\Th} +\theta(\slashed{v},\slashed{v})+\slashed{v}\big(-v^{\Th}+c\big)\right]-\frac{1}{2}\Psi_0 \cdot c\Th(s),\\
			\label{eq:Euler constraints Cb}L(w)=&-2c\Th(w)-\frac{1}{2}  \left[c\slashed{\rm div}(\slashed{v})+c{\rm tr} _{\slashed{g}} \theta \cdot v^{\Th}+2c\kappa^{-1}\slashed{v}(\kappa) -\slashed{v}(-v^{\Th}+c)-\theta(\slashed{v},\slashed{v})\right]\\
			\notag&+\frac{1}{2}\Psi_0\cdot c\Th(s),\\
			L(s)=&-c\Th(s),
		\end{align}
	\end{subequations}
	where $\Psi_0(\rho,s)=\int \frac{1}{\rho}\frac{\partial c}{\partial s} d\rho-\frac{1}{c\rho}\frac{\partial p}{\partial s}$.
	\begin{remark}
		It is crucial to observe that the righthand side of the first equation does not have the transversal derivative $\Th$ on $w$ or $\wb$. In comparison, the righthand sides of the second and the third  equations contain $\Th(w)$ and $\Th(s)$. 
	\end{remark}

	\subsection{The hypersurface $C_0$ is characteristic}
	
	To study the condition for $C_0$ being characteristic, we first introduce the concept of the speed functions defined on  $C_0$. Given a point $p\in C_0\cap \Sigma_t=S_{t,0}$,  we choose a new Cartesian coordinate system $(y^1,y^3,y^3)$ on $\Sigma_t$ so that $p=(0,0,0)$ and $T_pS_{t,0}=\{y^3=0\}$. We require that the $\partial_3$ direction is opposite to the $\Th$ direction and that ${\partial_i} (i=1,2,3)$ is a set of standard orthogonal basis on $\Sigma_t$. We remark that $y^3$ is then uniquely determined and $(y^1,y^2)$ is unique up to a rotation. Therefore, $S_{t,0}\subset \Sigma_t$  can be locally represented uniquely as a graph 
	\[y^3=\phi(t,y^1,y^2),\]
	where $\phi$ depends smoothly on $t$. Moreover, we have 
	\[\frac{\partial \phi}{\partial y^1}(t,0,0)=\frac{\partial \phi}{\partial y^2}(t,0,0)=0.\]
	We define \emph{the speed of $C_0$ at $p$} to be 
	\[\mathbf{V}(p):=\frac{\partial \phi}{\partial t}(t,0,0).\] 
	In fact, there is a geometric interpretation of the function $\mathbf{V}$. First of all,  exists a unique $2$-dimension plane $H$ passing through $p$ in $\mathbb{R}^4$ spanned by the $\Th(p)$ and $\frac{\partial}{\partial x^0}$. Therefore, $H\cap C_0$ is a smooth curve in $\mathbb{R}^4$. If we parametrize the curve by $x^0$, the speed of the curve is exactly $\mathbf{V}(p)$. This shows that the speed function defined on $C_0$ is independent of the choice of the above coordinates.  
	
	The tangent space $T_pC_{0}$ is spanned by
	\[V_0=\partial_t\phi \partial_3+\partial_0, \ V_1=\partial_1 \phi\partial_3+\partial_1, \  V_2=\partial_2 \phi\partial_3+\partial_2,\]
	where $\partial_i=\frac{\partial}{\partial y^i}$.
	Thus, the hypersurface $C_0$ is null at the point $p$ is equivalent to 
	\[\det\left( {\begin{array}{ccc}
			g(V_0,V_0) & g(V_0,V_1) &g(V_0,V_2)  \\
			g(V_1,V_0) & g(V_1,V_1) &g(V_1,V_2)  \\
			g(V_2,V_0) & g(V_2,V_1) &g(V_2,V_2)  
	\end{array} } \right)=0.\]
	Since
	\[\begin{cases}
		&g(V_0,V_0)=-c^2+|v|^2-2v^3\partial_t\phi+(\partial_t\phi)^2,\\
		&g(V_0,V_i)=\partial_t\phi\partial_i\phi-v^3\partial_i\phi-v^i,\\
		&g(V_i,V_i)=1+(\partial_i\phi)^2,\\
		&g(V_1,V_2)=\partial_1\phi \partial_2\phi,
	\end{cases}
	\]
	we compute that
	\begin{align*}
		&\det\left( {\begin{array}{ccc}
				g(V_0,V_0) & g(V_0,V_1) &g(V_0,V_2)  \\
				g(V_1,V_0) & g(V_1,V_1) &g(V_1,V_2)  \\
				g(V_2,V_0) & g(V_2,V_1) &g(V_2,V_2)  
		\end{array} } \right)\\
		=&\det\left( {\begin{array}{ccc}
				-c^2+|v|^2-2v^3\partial_t\phi+(\partial_t\phi)^2 & \partial_t\phi\partial_1\phi-v^3\partial_1\phi-v^1 &\partial_t\phi\partial_2\phi-v^3\partial_2\phi-v^2  \\
				\partial_t\phi\partial_1\phi-v^3\partial_1\phi-v^1 & 1+(\partial_1\phi)^2&\partial_1\phi \partial_2\phi \\
				\partial_t\phi\partial_2\phi-v^3\partial_2\phi-v^2  & \partial_1\phi \partial_2\phi &1+(\partial_2\phi)^2 
		\end{array} } \right)\\
		=&\det\left( {\begin{array}{ccc}
				-c^2+|v|^2-2v^3\partial_t\phi+(\partial_t\phi)^2 &  -v^1 & -v^2  \\
				-v^1 & 1&0 \\
				-v^2  &0 &1
		\end{array} } \right)\\
		=&-c^2+(v^3)^2-2v^3\partial_t\phi+(\partial_t\phi)^2=(v^3-\partial_t\phi)^2-c^2.
	\end{align*}
	Since $T_pS_{t,0}=\{y^3=0\}$, we have $v^3 =-v^{\Th}$. Therefore, the hypersurface $C_0$ is (outgoing) null at the point $p$ is equivalent to the following condition:
	\[-v^{\Th}+c=\partial_t\phi.\]
	We conclude that
	\begin{lemma}\label{lem: being characteristic}
		The following two statements are equivalent:
		\begin{itemize}
			\item[1)] The  hypersurface $C_0$ is an outgoing null hypersurface with respect to the acoustical metric $g$;
			\item[2)] The data $(\rho,v,s)$ on $C_0$ satisfies
			\begin{equation}\label{compatible condition 1}
				-v^{\Th}+c=\mathbf{V}
			\end{equation}
			where $c$ is the sound speed determined by $(\rho,v,s)$ on $C_0$.
		\end{itemize}
	\end{lemma}
	\begin{corollary}\label{coro: data on S00}


		Given initial data $(\mathring{\rho},\mathring{v},\mathring{s})$ inside $S_{0,0}$ on $\Sigma_0$, for $C_0$ being an outgoing null hypersurface with respect to the acoustical metric $g$, the following compatibility condition must hold on $S_{0,0}$:
		\begin{equation}\label{compatible condition 1 bis}
			-\mathring{v}^{\Th}+\mathring{c}=\mathbf{V}.
		\end{equation}
	\end{corollary}
	
	\begin{remark}
		We consider the Goursat problem with two characteristic hypersurfaces  $C_0$ and $\Cb_0$, as in Remark \ref{rem: pure goursat}. If $C_0\cap \Cb_0 \subset \Sigma_0$, then $\mathring{v}^{\Th}$ and $\mathring{c}$ on $C_0\cap \Cb_0$ are completely determined by the speed functions $\mathbf{V}$ and $\underline{\mathbf{V}}$ (of $\Cb_0$).  This is simple because
		\[\begin{cases}&-\mathring{v}^{\Th}+\mathring{c}=\mathbf{V},\\
			&-\mathring{v}^{\Th}-\mathring{c}=\underline{\mathbf{V}}.
		\end{cases}\]
		Moreover, if the entropy $s$ is also given on the bifurcate sphere $C_0\cap \Cb_0$, the two Riemann invariants $w$ and $\wb$ are also completely determined on $C_0\cap \Cb_0$.
	\end{remark}

	\subsection{The precise statement of the main theorem}

	\subsubsection{Determine $\Th^i,\slashed{g}$ and $\theta$}\label{sec:order 0 1}
	
	We recall that the restriction of the acoustical metric $g$ on each time slice $\Sigma_t$ is exactly the Euclidean metric $\sum_{i=1}^3 dx^i \otimes dx^i$. Therefore, since $S_{t,0}=C_0\cap \Sigma_t$ is given, the components of its unit normal $\Th^i$ ($i=1,2,3$) are known along $C_0$. Since $\slashed{g}$ and $\theta$ are the first and second fundamental forms of $S_{t,0}$ as an embedded hypersurface of $\Sigma_t$ respectively, $\slashed{g}$ and $\theta$ is also known along $C_0$. To summarize,  $\Th^i,\slashed{g}$ and $\theta$ on $C_0$ are already determined by the hypersurface $C_0$.

	\subsubsection{The free data $s$  and $\slashed{v}$ }\label{sec:order 0 2}
	
	We now give the initial data $s$ and $\slashed{v}$ on  $C_0$.
	
	Since we have already prescribed initial data $(\mathring{\rho},\mathring{v},\mathring{s})$ on $\Sigma_0$, all the derivatives of $(\rho, v, s)$ have been determined on $S_{0,0}$. On the other hand, since $L=\partial_t+\displaystyle\sum_{i=1}^3 \left(v^i-c\Th^i\right)\frac{\partial}{\partial x^i}$, therefore, $L\big|_{S_{0,0}}$ is already given by the  initial data $(\mathring{\rho},\mathring{v},\mathring{s})$ on $\Sigma_0$.  Therefore, for $f=s, \slashed{v}^1, \slashed{v}^2, \slashed{v}^3$ where $\slashed{v}=\sum_{i=1}^3 \slashed{v}^i\frac{\partial}{\partial x^i}$,  $L^k(f)\big|_{S_{0,0}}$ are all known for all $k\geqslant 0$. We denote these quantities as
	\[L^k(\mathring{s})\big|_{S_{0,0}},L^k(\mathring{\slashed{v}^1})\big|_{S_{0,0}},L^k(\mathring{\slashed{v}^2})\big|_{S_{0,0}},L^k(\mathring{\slashed{v}^3})\big|_{S_{0,0}}.\]
	
	We fix  an arbitrary smooth function $s\in C^\infty(C_0)$ with the following compatible condition
	\begin{equation}\label{compatible condition 2}
		L^k(s)\big|_{S_{0,0}}=L^k(\mathring{s})\big|_{S_{0,0}}, \ \ k\geqslant 0.
	\end{equation}
	We fix an arbitrary smooth vector field $\slashed{v}$ which is tangential to all $S_{t,0}$ where $t\in [0,T]$ so that 
	\begin{equation}\label{compatible condition 3}
		L^k(\slashed{v}^i)\big|_{S_{0,0}}=L^k(\mathring{\slashed{v}^i})\big|_{S_{0,0}}, \ \ i=1,2,3, \ k\geqslant 0.
	\end{equation}
	The conditions  \eqref{compatible condition 2} and   \eqref{compatible condition 3} are necessary conditions for the solution  of the Euler equations being smooth near $S_{0,0}$.

	\subsubsection{Determine $w$ and $\wb$}\label{sec:order 0 3}
	
	We show that given $\slashed{v}$ and $s$ on $C_0$, we can now determine  $w$ and $\wb$ on $C_0$. Therefore, we have obtained the values of $v, \rho$ and $s$ on $C_0$. The key idea is to make use the first equation of \eqref{eq:Euler transformed final 4}. Since we have already known the value of $\mathbf{V}$ along $C_0$, in view of Lemma \ref{lem: being characteristic}, we have
	\[L(\wb)=-\frac{1}{2} \left[c\slashed{\rm div}(\slashed{v})+c{\rm tr} _{\slashed{g}} \theta \cdot v^{\Th} +\theta(\slashed{v},\slashed{v})+\slashed{v}\big(\mathbf{V}\big)\right]-\frac{1}{2}\Psi_0 \cdot c\Th(s).\]
	According to the third equation of \eqref{eq:Euler transformed final 4},  we replace $c\Th(s)$ in the last term by $L(s)$. We also use $v^\Th=c-\mathbf{V}$. Thus,
	\begin{equation}\label{eq:aux Lwb}
		L(\wb)=-\frac{1}{2} \left[c\slashed{\rm div}(\slashed{v})+{\rm tr} _{\slashed{g}} \theta \cdot c^2-c{\rm tr} _{\slashed{g}} \theta \cdot \mathbf{V}+\theta(\slashed{v},\slashed{v})+\slashed{v}\big(\mathbf{V}\big)\right]+\frac{1}{2}\Psi_0 \cdot L(s).
	\end{equation}
	Since the values of $\slashed{v}, s, \mathbf{V}$ and $\theta$ on $C_0$ are known and also the intrinsic geometry of $S_{t,0}$ is given,   the values of $\slashed{\rm div}(\slashed{v}), {\rm tr} _{\slashed{g}} \theta, \theta(\slashed{v},\slashed{v}), \slashed{v}\big(\mathbf{V}\big)$ and $L(s)$ are known. It remains to deal with the terms $c$ and $\Psi_0$ on the righthand side of the above equation. 
	
	Be the definition \eqref{def: Riemann invariants} of Riemann invariants, we have $\Phi_0=w+\wb$. Since $s$ is already known, we obtain that $c=c(\rho)=c(w,\underline{w})$, i.e., $c$ is a function of $w$ and $\wb$.  Similarly, we have  $v^\Th=v^{\Th}(w,\wb)$ and $\Psi_0=\Psi_0(w,\underline{w})$. On the other hand, from  $\big(-v^{\Th}+c\big)(w,\wb)=\mathbf{V}$, we then can make the assumption that  $w$ is a function of $\wb$. 
	
	\begin{remark}
		The assumption that $\wb$ determines $w$ through the equation  $\big(-v^{\Th}+c\big)(w,\wb)=\mathbf{V}$ depends on the the equation of the state of the gas. The assumption holds for most of the interesting cases in physics. For example, for polytropic gas, i.e.,
		\[p(\rho,s)=\frac{1}{\gamma}\rho^\gamma A(s),\]
		where $A(s)$ is a smooth function of $s$ and $\gamma\in (1,3)$, we have
		\[
		\begin{cases}
			\wb=\frac{1}{2}\left(\frac{2}{\gamma-1}c- v^{\Th}\right), \\ 
			w=\frac{1}{2}\left(\frac{2}{\gamma-1}c+v^{\Th}\right).
		\end{cases}
		\]
		In this case, $w$ can be written as a linear function of $\wb$ provided the value of $-v^{\Th}+c$.
		
	\end{remark}
	
	Therefore, the only unknown function on the righthand side of \eqref{eq:aux Lwb} is $\wb$. Thus, $c=c(\wb)$ and $\Psi_0=\Psi_0(\wb)$. This leads to the following ODE system along each null geodesic  for $\underline{w}$:
	\begin{equation}\label{eq:final equation along wb}
		\begin{cases}
			L(\wb)=-\frac{1}{2} \left[c \slashed{\rm div}(\slashed{v})+{\rm tr} _{\slashed{g}} \theta \cdot c ^2-c {\rm tr} _{\slashed{g}} \theta \cdot \mathbf{V}+\theta(\slashed{v},\slashed{v})+\slashed{v}\big(\mathbf{V}\big)\right]+\frac{1}{2}\Psi_0 \cdot L(s),\\
			\wb|_{S_{0,0}}=\mathring{\wb},
		\end{cases}
	\end{equation}
	where $\mathring{\wb}$ is given by the  initial data $(\mathring{\rho},\mathring{v},\mathring{s})$ inside $S_{0,0}$ on $\Sigma_0$.  We remark that the righthand side of the above ODE is a smooth function of $\wb$. On the other hand, it is important to observe that, since $L=\partial_t+\displaystyle\sum_{i=1}^3 \left(v^i-c\Th^i\right)\frac{\partial}{\partial x^i}$, the vector field $L$ is also completely determined by $\mathbf{V}=-v^{\Th}+c$.   We then can solve the  ODE system \eqref{eq:final equation along wb} for $\wb$ to obtain the values $\underline{w}$ on $C_0$. This also gives the values of  $w$ on $C_0$.  
	
	\begin{remark}\label{rem: local existence}
		We make the assumption that the solution of the ODE system \eqref{eq:final equation along wb} exists for $t\in [0,T]$. In fact, there are several reasons that we may restrict to the local in time solutions of \eqref{eq:final equation along wb}.  
		
		First of all, the right hand side of \eqref{eq:final equation along wb} is in general nonlinear in $\wb$. For example, for the isentropic and polytropic gases, it reduces to
		\[
		L(\wb)=-\frac{1}{2} \left[c \slashed{\rm div}(\slashed{v})+{\rm tr} _{\slashed{g}} \theta \cdot c ^2-c {\rm tr} _{\slashed{g}} \theta \cdot \mathbf{V}+\theta(\slashed{v},\slashed{v})+\slashed{v}\big(\mathbf{V}\big)\right].\]
		The ${\rm tr} _{\slashed{g}} \theta \cdot c ^2$ term is then quadratic in $\wb$. Wether the solution can be extended to $[0,T]$ may depend on the sign of ${\rm tr} _{\slashed{g}} \theta$. 
		
		Secondly, we require $c>0$ because the sound speed is always positive for physical solutions. Since $c$ can be determined by $\wb$, we may have to stop the solution after a local existence to avoid vacuum for the gases.
		
		Thirdly, the inverse density $\kappa$ is also a positive function. Once $\wb$ is given, we will show that $\kappa$ can also be determined by an ODE system on $C_0$. We may have to stop the solution after a local existence to avoid shocks (where $\kappa=0$) for the gases.
		
		All the above requirements hold for small $T$ because of the continuous dependence of the solution of ODEs on the initial data.
	\end{remark}

	\subsubsection{The main theorem}
	
	For a tangential vector field $X$ on $C_0$, if for all $t\in [0,T]$, $X$ is tangential to $S_{t,0}$, we say that $X$ is a \emph{horizontal vector field} on $C_0$. In view of \eqref{compatible condition 1 bis}, \eqref{compatible condition 2} and \eqref{compatible condition 3}, we state the main theorem of the paper as follows:
	\begin{theorem}\label{thm:main theorem}
		
		Given an initial cone  $C_0$ in $\mathcal{D}$, given initial data $(\mathring{\rho},\mathring{v},\mathring{s})$ on $\Sigma_0$ inside $S_{0,0}$ so that 
		\[-\mathring{v}^{\Th}+\mathring{c}=\mathbf{V}\]
		hold on $S_{0,0}$. For all smooth function $f$ and horizontal vector field $X$ on  $C_0$  with the compatible conditions 
		\begin{equation}\label{compatible conditions}
			\begin{cases}
				&L^k(f)\big|_{S_{0,0}}=L^k(\mathring{s})\big|_{S_{0,0}}, \ \ k\geqslant 0,\\
				&L^k(X^i)\big|_{S_{0,0}}=L^k(\mathring{\slashed{v}^i})\big|_{S_{0,0}}, \ \ i=1,2,3, \ k\geqslant 0.
			\end{cases}
		\end{equation}
		there exists smooth functions $(\rho,v,s)$ on $C_0$ as the initial data for Euler equations so that
		\begin{itemize}
			\item[1)] $C_0$ is a characteristic hypersurface with respect to $(\rho,v,s)$;
			\item[2)]  On $C_0$, we have $s=f,  \vs=X$.
		\end{itemize}
		Moreover, all possible derivatives of $(\rho,v,s)$ along $C_0$ are also determined by $f$ and $X$. 
	\end{theorem}
	\begin{remark}
		In view of Remark \ref{rem: local existence}, in the above theorem, we assumed that the solution to the ODE systems in Remark \ref{rem: local existence} can be solved up to $[0,T]$.
	\end{remark}

	In the rest of the paper, we show that given the data $(\rho, v,s)$ as in the above theorem, we can determine all possible derivatives of $(\rho,v,s)$ along $C_0$ up to any given order. This completes the construction of the initial data on $C_0$.

	\section{Determine all the jets of the data along $C_0$}
	
	\subsection{Basic variables and orders}
	
	In order to describe the structure of the Euler equations on $C_0$,  we introduce the proper algebraic language to capture the natural quantities appearing in the Euler equations and the structure equations of the acoustical geometry.
	
	First of all, we consider the $3\times 3$ transformation matrix $\Omega=(\Omega_{ij})$ between the frame $(X_1,X_2,\Th)$ and $(\partial_1,\partial_2,\partial_3)$ on each $\Sigma_t$. By definition, we have
	\begin{equation*}
		\begin{pmatrix}
			X_1\\
			X_2\\
			\Th
		\end{pmatrix}=\Omega\cdot \begin{pmatrix}
			\partial_1\\
			\partial_2\\
			\partial_3
		\end{pmatrix}=\begin{pmatrix}
			X_1^1 & X_1^2 & X_1^3\\
			X_2^1 & X_2^2 & X_2^3\\
			\Th^1 & \Th^2 & \Th^3
		\end{pmatrix} \begin{pmatrix}
			\partial_1\\
			\partial_2\\
			\partial_3
		\end{pmatrix}.
	\end{equation*}
	Thus, we have
	\begin{equation}\label{eq: Omega prime}
		\begin{pmatrix}
			\partial_1\\
			\partial_2\\
			\partial_3
		\end{pmatrix}=\Omega'\cdot \begin{pmatrix}
			X_1\\
			X_2\\
			\Th
		\end{pmatrix}.
	\end{equation}
	where $\Omega'=\Omega^{-1}$ and its entries $\Omega'_{ij}$ can be expressed algebraically in terms of the entries of $\Omega$.
	
	We define the set
	\begin{align*}
		&\mathfrak{X}_{0}= \big\{\rho,\rho^{-1}, v^i,s, X_A^i, \Th^i,\Omega'_{ij},\kappa,\kappa^{-1}, c,c^{-1},\Phi_0(\rho,s), \Psi_0(\rho,s)
		\big|A=1,2;i,j=1,2,3\big\}.
	\end{align*}
	We racall that
	\[ 
	\begin{cases}&\Phi_0(\rho,s)=\int \frac{c(\rho,s)}{\rho} d\rho,\\
		&\Psi_0(\rho,s)=\int \frac{1}{\rho}\frac{\partial c}{\partial s} d\rho-\frac{1}{c\rho}\frac{\partial p}{\partial s},\\
		&X_A=\sum_{i=1}^3X_A^i\partial_i,\\
		&\Th=\sum_{i=1}^3\Th^i\partial_i.
	\end{cases}\]
	We also introduce the following set of differential operators
	\[\mathfrak{D}=\{L,T, X_1,X_2\}.\]
	Given a positive integer $n$, we define the following set of {\bf order $n$} objects:
	\[\mathfrak{X}_n=\big\{ (\mathfrak{d}_1\circ\mathfrak{d}_2\circ\cdots \circ\mathfrak{d}_n)(x)\big|x\in \mathfrak{X}_0, \mathfrak{d}_i \in \mathfrak{D}, 1\leqslant i\leqslant n\big\}.\]
	We also define the {\bf $T$-order} of the objects in $\mathfrak{X}_n$. Let $y=(\mathfrak{d}_1\circ\mathfrak{d}_2\circ\cdots \circ\mathfrak{d}_n)(x)$ where $x\in \mathfrak{X}_0,\mathfrak{d}_i \in \mathfrak{D}$ with $1\leqslant i\leqslant n$. Let $k$ be the number of $T$'s appearing in the set $\{\mathfrak{d}_1,\mathfrak{d}_2,\cdots,\mathfrak{d}_n\}$, we define the $T$-order of $y$ to be 
	\[{\rm ord}_{T}(y)=\begin{cases}
		k, \ \ &\text{if $x\neq \kappa, \kappa^{-1}$};\\
		k+1, \ \ &\text{if $x= \kappa$ or $\kappa^{-1}$}.
	\end{cases}
	\]
	In other words, the $T$-order counts the number of $T$-derivatives acting on an object and we regard $\kappa$ as an object of $T$-order equal to $1$.
	
	We now introduce the polynomial ring $\mathbb{R}[\mathfrak{X}_{\leqslant n}]$ which is the set of all $\mathbb{R}$-coefficients polynomials with unknowns from $\mathfrak{X}_{\leqslant n}$. We write a object from $\mathbb{R}[\mathfrak{X}_{\leqslant n}]$ as ${\mathscr{P}}_n$ as a schematic expression. The following examples  help to elucidate the meaning of  the symbol $\mathscr{P}_n$:
	
	\begin{itemize}
		\item For $j=1,2,3$, $\vs^j=v^j-\left(\sum_{i=1}^3v^i\Th^i\right)\Th^j$. This is an object in $\mathbb{R}[\mathfrak{X}_{\leqslant 0}]$. Therefore,  we write $\vs=\mathscr{P}_0$ in the schematic form.
		
		\item For $\wb=\frac{1}{2}\left(\Phi_0- v^{\Th}\right) =\frac{1}{2}\left[\Phi_0-\left(\sum_{i=1}^3v^i\Th^i\right)\right]$, we also have  $\wb=\mathscr{P}_0$. Similarly, $w=\mathscr{P}_0$.
		
		\item Since $ck_{ij}=\frac{1}{2}(\partial_iv^j+\partial_jv^i)$, we have
		\[k_{ij}=\frac{1}{2c}\left(\big(\Omega'_{i1}X_1+\Omega'_{i2}X_2+\Omega'_{i3}\Th\big)(v^j)+\big(\Omega'_{j1}X_1+\Omega'_{j2}X_2+\Omega'_{j3}\Th\big)(v^i)\right)=\mathscr{P}_1.\]
		It is also straightforward to check that $k(X_A,X_B), k(X_A,\Th),k(\Th,\Th)$ are also $\mathscr{P}_1$, where $A,B=1,2$. 
		
		\item Since $\theta_{AB}=g(\nabla_{X_A}\Th,X_B)=\sum_{i=1}^3 X_A(\Th^i)X_B^i$, we have $\theta_{AB}=\mathscr{P}_1$. By $\chi_{AB}=\chi(X_A,X_B)=c(k_{AB}-\theta_{AB})$, we also have $\chi_{AB}=\mathscr{P}_1$.

	\end{itemize}
	We remark that, in different situations, the polynomial ${\mathscr{P}}_n$ may change to another polynomial in $\mathbb{R}[\mathfrak{X}_{\leqslant n}]$  but this will not affect the proof.  We also define the order of ${\mathscr{P}}_n$ as ${\rm ord}({\mathscr{P}}_n)=n$. For $n\leqslant m$, we can regard ${\mathscr{P}}_n$ as ${\mathscr{P}}_m$.

	For a  monomial in $\mathbb{R}[\mathfrak{X}_{\leqslant n}]$, it is a product of the term of the type $y=(\mathfrak{d}_1\circ\mathfrak{d}_2\circ\cdots \circ\mathfrak{d}_l)(x)$ with $l\leqslant n$ and $x\in \mathfrak{X}_0$. We define its $T$-order as the maximum of all the possible ${\rm ord}_T(y)$'s. For a polynomial in $\mathbb{R}[\mathfrak{X}_{\leqslant n}]$ which is a $\mathbb{R}$-linear combination of monomials, its $T$-order is defined as the maximum of the $T$-orders of theses monomials. We write an object from $\mathbb{R}[\mathfrak{X}_{\leqslant n}]$ with $T$-order less or equal to $k$ as ${\mathscr{P}}_{n,k}$ as a schematic expression. We notice that $k$ can be greater than $n$. For example, $\kappa ={\mathscr{P}}_{0,1}$. We also give some examples to clarify the notations:
	\begin{itemize}
		\item $w, \wb$ and $\vs$ are $\mathscr{P}_{0,0}$. 
		
		\item In view of the following computations
		\[k_{ij}=\frac{1}{2c}\left(\big(\Omega'_{i1}X_1+\Omega'_{i2}X_2+\Omega'_{i3}\Th\big)(v^j)+\big(\Omega'_{j1}X_1+\Omega'_{j2}X_2+\Omega'_{j3}\Th\big)(v^i)\right)\]
		and the fact that $\Th=\kappa^{-1} T$, we see that $k_{ij}=\mathscr{P}_{1,1}$. 
		
		\item In view of $ck_{ij}=\frac{1}{2}(\partial_iv^j+\partial_jv^i)$,
		\[k_{AB}=\frac{1}{2c}\left(X_A^iX_B^j\partial_iv^j+X_A^iX_B^j\partial_jv^i\right)=\frac{1}{2c}\left(X_B^jX_A(v^j)+X_A^iX_B(v^i)\right).\]
		We see that $k_{AB}=\mathscr{P}_{1,0}$. Its $T$-order is less than ${\rm ord}_T(k_{ij})$.
		
		\item Since $\theta_{AB}=g(\nabla_{X_A}\Th,X_B)=\sum_{i=1}^3 X_A(\Th^i)X_B^i$, we have $\theta_{AB}=\mathscr{P}_{1,0}$. By $\chi_{AB}=\chi(X_A,X_B)=c(k_{AB}-\theta_{AB})$, we also have $\chi_{AB}=\mathscr{P}_{1,0}$.

	\end{itemize}

	The introduction of the symbols $\mathscr{P}_n$ or $\mathscr{P}_{n,k}$ allows us to use the following schematic computations to simplify the notations. By definition, we have
	\begin{equation}\label{eq:schematic algebraic rules}
		\begin{cases}
			& {\mathscr{P}}_n + {\mathscr{P}}_m ={\mathscr{P}}_{\max(m,n)},\\ 
			&{\mathscr{P}}_n \cdot {\mathscr{P}}_m ={\mathscr{P}}_{\max(m,n)},\\
			& {\mathscr{P}}_{n,k} + {\mathscr{P}}_{m,l} ={\mathscr{P}}_{\max(m,n),\max(k,l)},\\ 
			&{\mathscr{P}}_{n,k} \cdot {\mathscr{P}}_{m,l} ={\mathscr{P}}_{\max(m,n),\max(k,l)}.
		\end{cases}
	\end{equation}
	For $\mathfrak{d} \in \{L,X_1,X_2\}$,  we have
	\begin{equation}\label{eq:schematic derivatives}
		\begin{cases}
			&\mathfrak{d}\left({\mathscr{P}}_n\right)={\mathscr{P}}_{n+1}, \\
			&\mathfrak{d}\left({\mathscr{P}}_{n,k}\right)={\mathscr{P}}_{n+1,k}.
		\end{cases}
	\end{equation}
	We also have
	\begin{equation}\label{eq:schematic T derivatives}
		\begin{cases}
			&\Th\left({\mathscr{P}}_n\right)={\mathscr{P}}_{n+1}, \\
			&\Th\left({\mathscr{P}}_{n,k}\right)={\mathscr{P}}_{n+1,k+1}.
		\end{cases}
	\end{equation}
	
	We now prove the main result of the paper, i.e. Theorem \ref{thm:main theorem}. The rest of the section is to determine the values of all $\mathscr{P}_n$ on $C_0$ for any fixed $n$.
	
	\subsection{Determine the terms with ${\rm ord}_T = 0$}

	\subsubsection{Determine the remaining ${\rm ord}_T = 0$ terms}
	According to the computations in Section \ref{sec:order 0 1}, Section \ref{sec:order 0 2} and Section \ref{sec:order 0 3}, we have already  obtained the values of  $\rho,\rho^{-1}, v^i,s,c,c^{-1},\Phi_0(\rho,s)$ and $\Psi_0(\rho,s)$ on $C_0$. Morover, the null vector field 
	\[L=B-c\Th= \partial_t+\displaystyle\sum_{i=1}^3 \left(v^i-c\Th^i\right)\frac{\partial}{\partial x^i}\]
	is known and it can be explicitly expressed in   the Cartesian coordinates.
	
	We turn to the rest of terms in $\mathfrak{X}_0$. By the construction of $\vartheta^A$, it is first given on $S_{0,0}=C_0\cap\Cb_0$ and then is extended to $C_0$ by the equation $L\vartheta^A=0$. We recall that $L$ is completely determined at this point. Therefore, we have $X_A=\frac{\partial}{\partial \vartheta^A}$ on $C_0$ by using $(t,\vartheta_1,\vartheta_2)$ as coordinates on $C_0$. In particular, we have determined all the  $X_A^i$'s where $A=1,2$ and $i,j=1,2,3$. Moreover, in view of the definition of $\Omega'$ for \eqref{eq: Omega prime} , we have also obtained the values of $\Omega'_{ij}$'s on $C_0$.
	
	\subsubsection{Conclusion}
	We have determined all quantities  of $\mathfrak{X}_{0}$ on $C_0$ except for $\kappa$ and $\kappa^{-1}$ which of $T$-order $1$. Since $L=\frac{\partial}{\partial t}$ and $X_A=\frac{\partial}{\partial \vartheta^A}$ in the $(t,\vartheta_1,\vartheta_2)$ coordinates on $C_0$, we have determined the values of all of the following terms whose $T$-order is $0$:
	\[\big\{ (\mathfrak{d}_1\circ\mathfrak{d}_2\circ\cdots \circ\mathfrak{d}_n)(x)\big|x\in \mathfrak{X}_0- \{\kappa,\kappa^{-1}\}, \mathfrak{d}_i \in \{L,X_1,X_2\}, 1\leqslant i\leqslant n\big\}.\]
	In view of the definition of $\mathscr{P}_{0}$, as long as there is no $\kappa$ appearing in the monomials of $\mathscr{P}_0$, its values is already determined on $C_0$. Moreover, we have also determined all the $\mathscr{P}_{n,0}$'s on $C_0$.

	\subsection{Determine the terms with ${\rm ord}_T = 1$}\label{SS:T=1}
	
	\subsubsection{Determine $\Th\rho,\Th v^i$ and $\kappa$}
	
	According to the second equation in \eqref{eq: Euler in Euler coordinates}, for $i=1,2,3$, we have
	\begin{equation*}
		Lv^i+c\Th v^i=-\rho^{-1} c^2 \partial_i \rho-\rho^{-1}\frac{\partial p}{\partial s} \partial_i s.
	\end{equation*}
	where $\partial_i=\frac{\partial}{\partial x^i}$. Thus, 
	\begin{align*}
		\Th v^i&=-\rho^{-1} c \partial_i \rho-c^{-1}\rho^{-1}\frac{\partial p}{\partial s} \partial_i s-c^{-1}Lv^i\\
		&=\mathscr{P}_{0,0}\cdot \partial_i \rho+\mathscr{P}_{0,0}\cdot\partial_i s+\mathscr{P}_{1,0}.
	\end{align*}
	By the third equation of \eqref{eq:Euler transformed final 4},  we have $\Th(s)=-c^{-1}L(s)=\mathscr{P}_{1,0}$. 
	\begin{remark}\label{rem: 1 derivative on s}
		In view of  \eqref{eq: Omega prime}, this also shows that for all $i=1,2,3$, $\partial_i s =\mathscr{P}_{1,0}$.  In other words, we have already determined all possible first derivatives of the entropy $s$ on $C_0$.
	\end{remark}
	Since the righthand side of the above equation does not contain terms with $\kappa$, we obtain that 
	\begin{align*}
		\Th v^i&=\mathscr{P}_{0,0}\cdot \partial_i \rho+\mathscr{P}_{1,0}.
	\end{align*}
	where we use the fact that $\mathscr{P}_{1,0}\cdot \mathscr{P}_{0,0}=\mathscr{P}_{1,0}$. By \eqref{eq: Omega prime}, we have 
	\[\partial_i\rho =\Omega'_{i1}\cdot X_1(\rho)+\Omega'_{i2}\cdot X_2(\rho)+\Omega'_{i3}\cdot \Th(\rho)=\mathscr{P}_{0,0}\cdot \Th(\rho)+\mathscr{P}_{1,0}.\]
	\begin{remark}
		The above argument shows that, for all smooth function $f$, schematically, we have
		\begin{equation}\label{eq: replace partail i f by null frame}
			\partial_i f =\mathscr{P}_{0,0}\cdot X_A(f)+\mathscr{P}_{0,0}\cdot \Th(f).
		\end{equation}
	\end{remark}
	Combining all the above computations, we conclude that 
	\begin{equation}\label{eq: T vi replaced by Trho}
		\Th v^i=\mathscr{P}_{0,0}\cdot \Th\rho+\mathscr{P}_{1,0}.
	\end{equation}
	Therefore, once we determine the value of $\Th\rho$ on $C_0$, we obtain the value of $\Th v^i$'s automatically. In the course of the proof, we will use this equation to replace all the $\Th(v^i)$'s by  $\Th\rho$ and by other terms which have already been determined on $C_0$.
	
	It is suffices to determine $\Th \rho$ and $\kappa$ on $C_0$. We rely on the wave equation \eqref{eq: box rho} for $\rho$. We first rewrite the right hand side of the wave equation in a schematic form. Since we have already obtain the value of $v^i$, we then have
	\[\frac{\rho}{c^2} \big({\rm div}(v)^2-\nabla_i v^j \nabla_j v^i\big)=\mathscr{P}_{1,0}\big((\Th \rho)^2+\Th \rho\big)+\mathscr{P}_{1,0},\]
	where we have used \eqref{eq: T vi replaced by Trho} to replace the $\Th(v^i)$'s. By using \eqref{acoustic in inverse matrix}, we have
	\[\frac{g(D\rho,D\rho)}{\rho}-\frac{g(Dc,D\rho)}{c}=\mathscr{P}_{1,0}\big((\Th \rho)^2+\Th \rho\big)+\mathscr{P}_{1,0}.\]
	We can then use Remark \ref{rem: 1 derivative on s} to the terms involving $s$ on the righthand side of \eqref{eq: box rho} except for the term $\triangle s$. Combining the above computations, \eqref{eq: box rho} is transformed into the following form:
	\begin{equation}\label{eq: box rho intermediate 1}
		\Box_g(\rho)= \mathscr{P}_{0,0}\cdot \triangle s+ \mathscr{P}_{1,0}\cdot \big((\Th \rho)^2+\Th \rho\big)+\mathscr{P}_{1,0}.
	\end{equation}
	We then use  \eqref{eq: wave in L T} to expand the left hand side of the above equation. Since for each $t$, the geometry of the embedded surface $S_{t,0}$ is known, we then have $\slashed{\triangle}_{\slashed{g}}(\rho)=\mathscr{P}_{2,0}$. Since $\rho$ is already known on $C_0$, we also have $L^2(\rho)=\mathscr{P}_{2,0}$. According to \eqref{eq: wave in L T}, we have
	
	\begin{equation*}
		\begin{split}
			\Box_g \rho&=-\frac{2}{\mu} LT(\rho)+\mathscr{P}_{2,0}-c^{-2}\kappa^{-1} L(\kappa)L(\rho)-c^{-1}\tr\slashed{k}\cdot L(\rho)-(\tr\slashed{k}-\tr\theta)\Th(\rho)-\frac{2}{\mu}\zeta^A X_A(\rho)\\
			&=-\frac{2}{\mu} LT(\rho)+\mathscr{P}_{2,0}+\mathscr{P}_{1,0}\kappa^{-1} L(\kappa)+\mathscr{P}_{1,0}\Th(\rho)-\frac{2}{\mu}\zeta^A X_A(\rho).
		\end{split}
	\end{equation*}
	where we have used the fact that $\tr\slashed{k}$ and $\tr\theta$ are $\mathscr{P}_{1,0}$.  By \eqref{eq: zeta in X and T}, we have
	\begin{align*}
		\frac{2}{\mu}\zeta^A X_A(\rho)&=\frac{2}{c}\slashed{g}^{AB}\big(ck(X_B, \Th)-X_B(c)\big)X_A(\rho)\\
		&=\frac{2}{c}\slashed{g}^{AB}\left(\sum_{j=1}^3\frac{1}{2}X_B(v^j)\Th^j+\sum_{i=1}^3\frac{1}{2}X_B^i\Th(v^i)-X_B(c)\right)X_A(\rho)\\
		&=\mathscr{P}_{1,0}+\mathscr{P}_{1,0}\sum_{i=1}^3\Th(v^i)\\
		&=\mathscr{P}_{1,0}+\mathscr{P}_{1,0}\cdot \Th(\rho).
	\end{align*}
	In the last step, we have used \eqref{eq: T vi replaced by Trho}. 
	
	For $L(\kappa)$, we use \eqref{eq: structure equation L kappa} to derive
	\begin{align*}
		L\kappa&=-\kappa\left(\frac{\partial c}{\partial \rho}\Th(\rho)+\frac{\partial c}{\partial s}\Th(s)\right)+\sum_{j=1}^3\Th^j T(v^j)\\
		&=\mathscr{P}_{1,0} \cdot  \kappa \cdot \Th(\rho)+\kappa\cdot\mathscr{P}_{1,0}.
	\end{align*}
	Therefore, we obtain that
	\[
	\Box_g \rho=-\frac{2}{\mu} LT(\rho)+\mathscr{P}_{1,0}\cdot \Th(\rho)+\mathscr{P}_{2,0}.
	\]
	We now compute
	\begin{align*}
		-\frac{2}{\mu} LT(\rho)&=-\frac{2}{c\kappa}L(\kappa)\Th(\rho)-\frac{2}{c}L(\Th(\rho))=-\frac{2}{c}L(\Th(\rho))+\mathscr{P}_{1,0}\left(\Th(\rho)+(\Th(\rho))^2\right).
	\end{align*}
	Therefore,
	\[
	\Box_g \rho=-\frac{2}{c} L\Th(\rho)+\mathscr{P}_{1,0}\left(\Th(\rho)+(\Th(\rho))^2\right)+\mathscr{P}_{2,0}.
	\]
	\begin{remark}
		Since $\Th(s)=-c^{-1}L(s)=\mathscr{P}_{1,0}$, we can use \eqref{eq: wave in L T} and  repeat the above argument to derive that
		\begin{equation}\label{eq: box s intermediate 3}
			\begin{split}
				\Box_g s&=-\frac{2}{c} L\Th(s)+\mathscr{P}_{1,0}\cdot\Th(\rho)+\mathscr{P}_{2,0}\\
				&=\mathscr{P}_{1,0}\cdot\Th(\rho)+\mathscr{P}_{2,0}.
			\end{split}
		\end{equation}
	\end{remark}
	Combined with \eqref{eq: box rho intermediate 1}, we obtain that
	\begin{equation}\label{eq: box rho intermediate 2}
		L\Th(\rho)
		= \mathscr{P}_{0,0}\triangle s+ \mathscr{P}_{1,0}\left((\Th \rho)^2+\Th \rho\right)+\mathscr{P}_{2,0}.
	\end{equation}
	It remains to compute $\triangle s$ in \eqref{eq: box rho intermediate 2}. We use two ways to compute $\Box_g s$. This first way is recorded in \eqref{eq: box s intermediate 3}. The second way is to use \eqref{eq:formula for box g phi} and the fact that $B(s)=0$. This leads to 
	\begin{align*}
		\Box_g s &=\frac{1}{c}\partial_i(c)\partial_i(s)+\triangle s=\triangle s+\frac{1}{c}\left(\frac{\partial c}{\partial \rho}\partial_i(\rho)+\frac{\partial c}{\partial s}\partial_i(s)\right)\partial_i(s)\\
		&=\triangle s+\mathscr{P}_{1,0}\Th(\rho)+\mathscr{P}_{1,0}
	\end{align*}
	Compare to \eqref{eq: box s intermediate 3}, we obtain that
	\[\triangle s=\mathscr{P}_{1,0}\cdot\Th(\rho)+\mathscr{P}_{2,0}.\]
	Therefore, \eqref{eq: box rho intermediate 2} can be rewritten as
	\[
	L\Th(\rho)
	= \mathscr{P}_{1,0}\left((\Th \rho)^2+\Th \rho\right)+\mathscr{P}_{2,0}.
	\]
	In view of the fact that
	\[
	L\kappa=\mathscr{P}_{1,0} \cdot \kappa \cdot \Th(\rho)+\kappa\cdot \mathscr{P}_{1,0},
	\]
	we finally obtain an ODE system for $\Th(\rho)$ and $\kappa$ as follows:
	
	\begin{equation}\label{eq: ODE system for Th rho and kappa}
		\begin{cases}
			L\left(\Th(\rho)\right)&=\mathscr{P}_{1,0}\left((\Th \rho)^2+\Th (\rho)\right)+\mathscr{P}_{2,0},\\
			L(\kappa)&=\mathscr{P}_{1,0} \cdot \kappa \cdot \Th(\rho)+\kappa\cdot \mathscr{P}_{1,0}.
		\end{cases}
	\end{equation}
	The initial value of $\Th(\rho)$ and $\kappa$ on $S_{0,0}$ are given by the initial data $(\mathring{\rho},\mathring{v},\mathring{s})$ on $\Sigma_0$ inside $S_{0,0}$. We then can solve the  ODE system \eqref{eq: ODE system for Th rho and kappa} to obtain the values of $\Th(\rho)$ and $\kappa$ on $C_0$. 
	
	\begin{remark}\label{rem: local existence 2}
		Similar to Remark \ref{rem: local existence}, we make the assumption that the solution of the ODE system \eqref{eq: ODE system for Th rho and kappa}  exists for $t\in [0,T]$. In fact, the right hand side of  \eqref{eq: ODE system for Th rho and kappa} is quadratic in $\kappa$ and $\Th(\rho)$ and this may prevent the solution existing on  the entire $[0,T]$ interval. Moreover, since $\kappa$ is positive and $\kappa$ is determined by the above ODE system on $C_0$, the solution may not exist once $\kappa$ reaches $0$. This is where  shocks may potentially form.  Once again, all the above requirements hold for small $T$ because of the continuous dependence of the solution of ODEs on the initial data. 
	\end{remark}

	\begin{remark}\label{rem: zeta formula}
		By \eqref{eq: zeta in X and T}, we have
		\begin{equation}\label{eq: zeta explicit}
			\zeta^A=\frac{ \kappa}{2}\slashed{g}^{AB}\left(\sum_{j=1}^3 X_B(v^j)\Th^j+\sum_{i=1}^3 X_B^i\Th(v^i)-2X_B(c)\right).
		\end{equation}
		Therefore, $\zeta^A$ can also be determined on $C_0$. Since $\eta^A = \zeta^A+X_A(c\cdot \kappa)$, $\eta^A$ is also determined on $C_0$.
	\end{remark}

	\subsubsection{Determine $\Th(\Th^i),\Th(X_A^i)$ and $\Th^2(s)$}
	
	We first determine $\Th^2(s)$ on $C_0$. By $L(s)=-c\Th(s)$ and the commutation formula \eqref{eq:commutator}, we have
	\begin{align*}
		LT(s)&=[L,T](s)-T(c\Th(s))\\
		&=-(\zeta^A+\eta^A)X_A(s)-\kappa \Th(c)\Th(s)-c\kappa \Th^2(s).
	\end{align*}
	Therefore,
	\begin{equation}\label{eq: T2 s}
		\Th^2(s)=-\frac{1}{c\kappa}\left(LT(s)+(\zeta^A+\eta^A)X_A(s)+\kappa \Th(c)\Th(s)\right).
	\end{equation}
	Since $T(s)=\kappa \Th(s)$ is already known on $C_0$, $LT(s)$ is then determined on $C_0$. Thus, all the terms on the righth and side are determined. This determines the value of $\Th^2(s)$ on $C_0$. Moreover, this shows that indeed $\Th^2(s)=\mathscr{P}_{2,1}$. 
	
	We then turn to $\Th(\Th^i)$. By \eqref{eq: T Th i}, we have
	\begin{equation}\label{eq: T Th i 2}
		\Th(\Th^i)=-\frac{1}{\kappa}\slashed{g}^{AB}X_B (\kappa)X_A^i.
	\end{equation}
	Therefore, $\Th(\Th^i)$'s are all determined on $C_0$.

	To derive the values of $\Th(X_A^i)$ on $C_0$, we write the formula $[L,X_A]=0$ in the Cartesian coordinates:
	\begin{equation}\label{eq:L XA i}
		L(X_A^i)=X_A(v^i-c\Th^i).
	\end{equation}
	By \eqref{eq:commutator 2} and the formula
	\[[\Th,X_A]=\sum_{i=1}^3\left(\Th(X_A^i)-X_A(\Th^i)\right)\partial_i,\]
	we compute that
	\begin{align*}
		L\Th(X_A^i)
		=&[L,\Th]X_A^i+[\Th,X_A](v^i-c\Th^i)+X_A\Th(v^i-c\Th^i)\\
		=&-\frac{1}{\kappa}\left(\zeta^A+\eta^A\right)X_A(X_A^i)-\frac{1}{\kappa}L(\kappa)\Th(X_A^i)\\
		&+\sum_{j=1}^3\left(\Th(X_A^j)-X_A(\Th^j)\right)\partial_j(v^i-c\Th^i)+X_A\Th(v^i-c\Th^i).
	\end{align*}
	Therefore,
	\begin{equation}\label{eq: L Th XAi}
		\begin{split}
			L\Th(X_A^i)&=-\frac{1}{\kappa}\left(\zeta^A+\eta^A\right)X_A(X_A^i)+X_A\Th(v^i-c\Th^i)\\
			&\ \ \ -\frac{1}{\kappa}L(\kappa)\Th(X_A^i)+\sum_{j=1}^3\left(\Th(X_A^j)-X_A(\Th^j)\right)\left(\Omega'_{j1}X_1+\Omega'_{j2}X_2+\Omega'_{j3}
			\Th\right)(v^i-c\Th^i).
		\end{split}
	\end{equation}
	
	Except for $\Th(X_A^i)$'s, all the others terms on the righthand side of the above equations have been determined. Therefore, the above equations can be regarded as an ODE system for the vector valued function $\left(\Th(X_A^1),\Th(X_A^2),\Th(X_A^3)\right)$. Since we have the value of $\Th(X_A^i)$'s on $S_{0,0}$,  we can solve the ODE system to obtain the values of $\Th(X_A^i)$'s on $C_0$.
	
	\begin{remark}\label{rem: local existence 3}
		In contrast to Remark \ref{rem: local existence} and Remark \ref{rem: local existence 2}, the right hand side of he ODE system \eqref{eq: L Th XAi} is linear in $\left(\Th(X_A^1),\Th(X_A^2),\Th(X_A^3)\right)$. Therefore, we can solve the system on the entire time interval $[0,T]$.
	\end{remark}
	
	\subsubsection{Determine the remaining ${\rm ord}_T = 1$ terms}
	According to the computations in the previous section, we have obtained the values of the $\Th$-derivative of all terms (except for $\kappa$ and $\kappa^{-1}$) from $\mathfrak{X}_0$ as well as the value of $\kappa$ along $C_0$. Thus, we have already obtained the value of $\mathscr{P}_{1,1}$ terms. Assume we have already obtained all the values of $\mathscr{P}_{n-1,1} (n\ge 2)$ terms, we now show that this determines the value of all $\mathscr{P}_{n,1}$ terms.

	By the definition of  $\mathscr{P}_{n,1}$, it suffices to consider a term $y$ of the form
	\[y=(\mathfrak{d}_1\circ\mathfrak{d}_2\circ\cdots \circ\mathfrak{d}_n)(x)\]
	with the restriction that
	\begin{itemize}
		\item[1)] If $x\in \mathfrak{X}_0- \{\kappa,\kappa^{-1}\}$,  there exists exactly one $j\leqslant n$, $\mathfrak{d}_j=\Th, \mathfrak{d}_i \in \{L,X_1,X_2\}, i\neq j, \ 1\leqslant i\leqslant n$;
		\item[2)]If $x\in \{\kappa,\kappa^{-1}\}$,  then $\mathfrak{d}_i \in \{L,X_1,X_2\}, 1\leqslant i\leqslant n$.
	\end{itemize}
	The second case is obvious. It suffices to consider the first case. If $j=n$, since $\Th(x)$ is already known, this case is also obvious. The following inductive argument essentially shows that the rest cases can be reduced to this one.
	
	By changing the frame $\{\partial_1,\partial_2,\partial_3\}$ into $\{X_1,X_2,\Th\}$, we can write the commutator $[\Th,X_A]$ schematically as follows:
	\[[\Th,X_A]=\sum_{i=1}^3\left(\Th(X_A^i)-X_A(\Th^i)\right)\partial_i=\mathscr{P}_{1,1}\cdot\mathfrak{d}\]
	Combined with \eqref{eq:commutator 2}, we conclude that for all $\mathfrak{d}',\mathfrak{d}''\in \{L,X_1,X_2,\Th\}$, 
	\begin{equation}\label{eq: commutator schematic}
		[\mathfrak{d'},\mathfrak{d}'']=\mathscr{P}_{1,1}\cdot\mathfrak{d}, 
	\end{equation}
	with $\mathfrak{d}\in \{L,X_1,X_2,\Th\}$.
	
	We now run an induction on $n$ and $j$ where $j\leqslant n$.  We have already discussed the case $j=n$  for all $n$. Assume that we can determine $\mathscr{P}_{n,1}$ on $C_0$ for $j$ and we will show that it is the case for $j-1$. Indeed, by the above commutator formula, we compute that
	\begin{align*}
		y&=(\mathfrak{d}_1\circ\cdots \circ \mathfrak{d}_{j-2}\circ \Th \circ\mathfrak{d}_{j} \circ \cdots \circ\mathfrak{d}_n)(x)\\
		&=(\mathfrak{d}_1\circ\cdots \circ \mathfrak{d}_{j-2}\circ [\Th, \mathfrak{d}_{j}] \circ \cdots \circ\mathfrak{d}_n)(x)+(\mathfrak{d}_1\circ\cdots \circ \mathfrak{d}_{j-2}\circ\mathfrak{d}_{j} \circ \Th \circ \cdots \circ\mathfrak{d}_n)(x)\\
		&=(\mathfrak{d}_1\circ\cdots \circ \mathfrak{d}_{j-2}\circ \big(\mathscr{P}_{1,1} \cdot\mathfrak{d}\big)\circ \cdots \circ\mathfrak{d}_n)(x)
		+(\mathfrak{d}_1\circ\cdots \circ \mathfrak{d}_{j-2}\circ\mathfrak{d}_{j} \circ \Th \circ \cdots \circ\mathfrak{d}_n)(x).
	\end{align*}
	The second term corresponds to the inductive hypothesis for $\Th$ appearing as $\mathfrak{d}_j$. For the first term, the total number of derivatives is $n-1$ and we will use the inductive hypothesis for $n-1$. Since all the $\mathfrak{d}_i$'s are from $\{L,X_1,X_2\}$, by Leibniz rule, the derivatives acting on $\mathscr{P}_{1,1}$ only generates the terms already determined on $C_0$. This finishes the inductive argument.

	\begin{remark}
		We have already determined the value of $\Th^2(s)$ on $C_0$.  The $T$-order of $\Th^2(s)$  is $1$.
	\end{remark}
	
	\subsection{Determine the terms with higher $T$-orders}\label{SS:T=k}
	
	In this section, we make the following inductive assumption:  we have already determined all the terms of the form $\mathscr{P}_{n,l}$ on $C_0$ for all $n$ and $l\leqslant k$ where $k\geqslant 1$. By the previous section, we also assume that $\Th^{k+1}(s)$ is also given on $C_0$. We will derive the values of $\mathscr{P}_{n,k+1}$ as well as the value of $\Th^{k+2}(s)$ on $C_0$.

	To simplify the notations, we  use $\mathfrak{d}^k$ to denote all possible differential operators $\mathfrak{d}_1\circ\mathfrak{d}_2\circ\cdots \circ\mathfrak{d}_k$ with $k\geqslant 1$, $\mathfrak{d}_i \in \{L,\Th, X_1,X_2\}$, $i\leqslant k$. We first derive a commutator formula for $[L,\mathfrak{d}^k]$. In view of \eqref{eq:commutator 2}, for $\mathfrak{d}\in \{L,X_1,X_2, \Th\}$,  we have the following the schematic formula 
	\[[L, \mathfrak{d}]={\mathscr{P}}_{1,1} \cdot \mathfrak{d}.\]
	Moreover, we observe that the derivative $T$ does not appear as $\mathfrak{d}$ on the right hand side of the above formula.
	\begin{lemma}
		For all $k\geqslant 1$ and all $\mathfrak{d}\in \{L,X_1,X_2, \Th\}$, we have
		\begin{equation}\label{eq:commutator:schematic L commute with d^k}
			[L,\mathfrak{d}^k]=\sum_{j=1}^k {\mathscr{P}}_{j,l} \cdot \mathfrak{d}^{k+1-j},
		\end{equation}
		where $l\leqslant k$ and the top order operator on the right hand side $\mathfrak{d}^{k} \neq \Th^k$.
	\end{lemma}
	\begin{proof}
		We prove by induction on $k$. It is clear that \eqref{eq:commutator:schematic L commute with d^k} holds for $k=1$. We assume that the statement holds for $k$, we now prove for the case $k+1$:
		\begin{align*}
			[L,\mathfrak{d}^{k+1}]x&=[L,\mathfrak{d}]\mathfrak{d}^k(x)+\mathfrak{d}\big([L,\mathfrak{d}^k]x\big)\\
			&={\mathscr{P}}_{1,1} \cdot \mathfrak{d}\big(\mathfrak{d}^k(x)\big)+\mathfrak{d}\big(\sum_{j=1}^k {\mathscr{P}}_{j,l} \cdot \mathfrak{d}^{k+1-j}\big)\\
			&={\mathscr{P}}_{1,1} \cdot \mathfrak{d}^{k+1}(x)+\sum_{j=1}^k \big(\underbrace{\mathfrak{d}({\mathscr{P}}_{j,l})}_{ \mathscr{P}_{j+1,l+1} }  \mathfrak{d}^{k+1-j}+{\mathscr{P}}_{j,l} \cdot \mathfrak{d}^{k+2-j}\big).
		\end{align*}
		From the induction hypothesis, it is clear that $\mathfrak{d}^{k+1} \neq T^{k+1}$.  This proves the formula \eqref{eq:commutator:schematic L commute with d^k}.
	\end{proof}

	\subsubsection{Determine $\Th^{k+1}(\rho),\Th^{k+1}(v^i)$ and $\Th^{k}(\kappa)$}
	We use \eqref{eq: ODE system for Th rho and kappa} to derive an ODE system for $\Th^{k+1}(\rho)$ and $\Th^{k}(\kappa)$. We first apply $\Th^k$ on both sides of the first equation of \eqref{eq: ODE system for Th rho and kappa}. This leads to
	\begin{align*}
		\Th^k L\left(\Th(\rho)\right)&=\sum_{a+b=k}\Th^a(\mathscr{P}_{1,0})\left(\Th^b\big[(\Th \rho)^2\big]+\Th^{b+1} (\rho)\right)+\Th^k\mathscr{P}_{2,0}\\
		&=\mathscr{P}_{k+1,k}\cdot\Th^{k+1} (\rho) +\mathscr{P}_{k+2,k}.
	\end{align*}
	where we have used the Leibniz rule for products. We then use \eqref{eq:commutator:schematic L commute with d^k} to commute derivatives for the lefthand side of the above equation. This shows
	\begin{align*}
		\Th^k L\left(\Th(\rho)\right)&= L\Th^k\left(\Th(\rho)\right)-[L,\Th^k]\Th(\rho)\\
		&= L\Th^{k+1} (\rho) +\sum_{j=1}^k {\mathscr{P}}_{j,l} \cdot \mathfrak{d}^{k+1-j}
		\Th(\rho)\\
		&= L\Th^{k+1} (\rho) +{\mathscr{P}}_{k+1,k}.
	\end{align*}
	In the last step, we used the fact that the derivative $\mathfrak{d}^{k+1-j}\neq \Th^k$. Putting the above calculations together, we obtain that
	\[L\Th^{k+1} (\rho)=\mathscr{P}_{k+1,k}\cdot\Th^{k+1} (\rho) +\mathscr{P}_{k+2,k}.\]
	Similarly, we can apply $\Th^{k}$ on both sides of the second equation of \eqref{eq: ODE system for Th rho and kappa} and proceed in the same manner. This eventually leads to
	\begin{equation}\label{eq: ODE system for  Th k Th rho and kappa}
		\begin{cases}
			L\left(\Th^{k+1}(\rho)\right)&=\mathscr{P}_{k+1,k}\cdot\Th^{k+1} (\rho) +\mathscr{P}_{k+2,k},\\
			L\left(\Th^k(\kappa)\right)&=\mathscr{P}_{k+1,k}\left(\Th^{k+1} (\rho) +\Th^{k} (\kappa)\right)+\mathscr{P}_{k+2,k}.
		\end{cases}
	\end{equation}
	The initial value of  $\Th^{k+1}(\rho)$ and $\Th^k(\kappa)$ on $S_{0,0}$ are given by the initial data $(\mathring{\rho},\mathring{v},\mathring{s})$ on $\Sigma_0$,  we solve the above ODE to obtain the values of $\Th^{k+1}(\rho)$ and $\Th^k(\kappa)$ on $C_0$. 
	
	\begin{remark}\label{rem: local existence 4}
		Similar to Remark \ref{rem: local existence 3} , the right hand side of he ODE system \ref{eq: ODE system for  Th k Th rho and kappa} is linear in $\Th^{k+1}(\rho)$ and $\Th^k(\kappa)$. Therefore, we can solve the system on the entire time interval $[0,T]$.
	\end{remark}

	To compute the value of $\Th^{k+1}(v^i)$'s on $C_0$, we apply $\Th^k$ to \eqref{eq: T vi replaced by Trho} to derive
	\begin{align*}
		\Th^{k+1}\left( v^i\right)&=\Th^{k}\left(\mathscr{P}_{0,0}\cdot \Th\rho\right)+\Th^{k}\mathscr{P}_{1,0}\\
		&=\mathscr{P}_{k,k}\cdot \Th^{k+1}(\rho)+\mathscr{P}_{k+1,k}.
	\end{align*}
	Since the terms appearing on righthand side have been already determined on $C_0$, this gives the value of $\Th^{k+1}(v^i)$.

	\subsubsection{Determine $\Th^{k+1}(\Th^i),\Th^{k+1}(X_A^i)$ and $\Th^{k+2}(s)$}
	By \eqref{eq: T2 s} and Remark \ref{rem: zeta formula}, we have
	\begin{align*} \Th\big[\Th(s)\big]&=-\frac{1}{c\kappa}\left(LT(s)+(\zeta^A+\eta^A)X_A(s)+\kappa \Th(c)\Th(s)\right)\\
		&=\mathscr{P}_{2,0}(L\Th(s)+\kappa^{-1}L(\kappa)\Th(s)+\Th(v^i)+ \Th(c)\Th(s)+\kappa^{-1} X_A(\kappa))+\mathscr{P}_{2,0}\\
		&=\mathscr{P}_{2,0}(L\Th(s)+\Th(v^i)+\Th(c)\Th(s))+\mathscr{P}_{2,0}\cdot\kappa^{-1}\left(L(\kappa)+X_A(\kappa)\right)+\mathscr{P}_{2,0}\\
		&=\mathscr{P}_{2,0}(L\Th(s)+\Th(v^i)+\Th(\rho)\Th(s)+\Th(s)^2)+\mathscr{P}_{2,0}\cdot\kappa^{-1}\left(L(\kappa)+X_A(\kappa)\right)+\mathscr{P}_{2,0},
	\end{align*}
where $i=1,2,3$. In the last step, we have used \eqref{eq: structure equation L kappa}. We can apply $\Th^k$ to the above equation to derive the following schematic expression
	\begin{align*} \Th^{k+2}(s)=&\mathscr{P}_{2+k,k}(\Th^k L\Th(s)+\Th^{k+1}(v^i)+\sum_{a+b=k}\Th^{a+1}(\rho)\Th^{b+1}(s)+\sum_{a+b=k-1}\Th^{a+1}(s)\Th^{b+2}(s))\\
	&+\mathscr{P}_{2+k,k}(\Th^k \kappa+\Th^k L\kappa +\Th^k X_A\kappa)+\mathscr{P}_{2+k,k},
	\end{align*}
where the terms $\mathscr{P}_{2+k,k}$ are known. Therefore, all the terms on the righthand side have already been determined except for $\Th^k L\Th(s)$, $\Th^k L\kappa$ and $\Th^k X_A\kappa$. By \eqref{eq:commutator:schematic L commute with d^k}, we have
	\begin{align*}
		\Th^k L \Th(s)&=L\Th^{k+1}(s)+[\Th^k,L]\Th(s)\\
		&=L\Th^{k+1}(s)+\sum_{j=1}^k {\mathscr{P}}_{j,l} \cdot \mathfrak{d}^{k+1-j}\Th(s),
	\end{align*}
	where $l\leqslant k$ and $\mathfrak{d}^{k+1-j}\neq \Th^k$. For $\Th^k X_A\kappa$, we present the following calculation: By \ref{eq: commutator schematic},
	
		\begin{align*}
		\Th^k X_A \kappa=& \Th^{k-1} X_A \Th \kappa+\Th^{k-1} [\Th,X_A] \kappa\\
		=&\Th^{k-1} X_A \Th \kappa+\Th^{k-1} \left(\mathscr{P}_{1,1} \mathfrak{d}\right) \kappa\\
		=&\Th^{k-2} X_A \Th^2 \kappa+\Th^{k-2} [\Th,X_A] \Th \kappa+\mathscr{P}_{k+1,k}\Th^k \kappa+\mathscr{P}_{k+1,k}\\
		=&\cdot\cdot\cdot\\
		=&X_A\Th^k \kappa+\mathscr{P}_{k+1,k}\Th^k \kappa+\mathscr{P}_{k+1,k}
	\end{align*}
	
	All the terms on the righthand side have already been determined. Similarly, we can obtain $\Th^k L\kappa$. Thus, for the right side of the expression of $\Th^{k+2}(s)$, all the terms have already been determined on $C_0$.  Combining all the above analysis, we have obtained the value of $\Th^{k+2}(s)$ on $C_0$.
	
	We turn to the value of $\Th^{k+1}(\Th^i)$. By \eqref{eq: T Th i}, we have
	\[\Th(\Th^i)=-\frac{1}{\kappa}\slashed{g}^{AB}X_B (\kappa)X_A^i=\mathscr{P}_{1,0}\kappa^{-1}\cdot X_B (\kappa).\]
	We can apply $\Th^k$ to the above equation to derive
	\[\Th^{k+1}(\Th^i)=\sum_{a+b=k}\mathscr{P}_{k+1,k}\Th^{a}(\kappa)\cdot \Th^{b}\left(X_B (\kappa)\right).\]
	All the terms on the right hand side have been determined on $C_0$, so is $\Th^{k+1}(\Th^i)$.
	
	It remains to compute the value of $\Th^{k+1}(X_A^i)$'s on $C_0$.  First of all, by \eqref{eq: zeta explicit} and \eqref{eq: structure equation L kappa}, we can write \eqref{eq: L Th XAi} as
	
	\begin{equation}
		\begin{split}
			L\Th(X_A^i)=&-\frac{1}{\kappa}\left(\zeta^A+\eta^A\right)X_A(X_A^i)+X_A\Th(v^i-c\Th^i)\\
			&-\frac{1}{\kappa}L(\kappa)\Th(X_A^i)+\sum_{j=1}^3\left(\Th(X_A^j)-X_A(\Th^j)\right)\left(\Omega'_{j1}X_1+\Omega'_{j2}X_2+\Omega'_{j3}
			\Th\right)(v^i-c\Th^i)\\
			=&\mathscr{P}_{1,0}\left(\kappa^{-1}X_A \kappa+\Th v^j+X_A\Th v^j+\Th s+X_A \Th s+\Th \Th^j+X_A \Th \Th^j\right)+\mathscr{P}_{2,0}\\
			&+\mathscr{P}_{1,0}\cdot \left(\Th v^j+\Th s+\Th \Th^j+1\right)\Th X_A^l,
		\end{split}
	\end{equation}
	where $j,l\in \{1,2,3\}$. By applying the operator $\Th^k$ on the above equation, we can regroup the lower order terms on the righthand side to derive
\begin{align*}
&\Th^k L\Th(X_A^i)\\
=&\mathscr{P}_{k+1,k}\cdot(\Th^k \kappa+\Th^k X_A\kappa+\Th^{k+1} v^j+\Th^k X_A \Th v^j+\Th^{k+1}s+\Th^k X_A \Th s+\Th^{k+1} \Th^j+\Th^k X_A \Th \Th^j)\\
+&\mathscr{P}_{k+1,k}\cdot (\Th^{k+1} v^j+\Th^{k+1} s+\Th^{k+1} \Th^j+1)\Th^{k+1} X_A^l+\mathscr{P}_{k+2,k}.
\end{align*}
For the terms $\Th^k X_A \Th v^j$, $\Th^k X_A \Th s$ and $\Th^k X_A \Th \Th^j$ on the right side of the above equation, we can obtain them in a similar way to calculating $\Th^k X_A \kappa$, so all the terms on the right side of the above equation except $\Th^{k+1} X_A^l$ have already been determined on $C_0$.
	

Using a method similar to calculating $\Th^k X_A \kappa$, we can obtain

	\begin{align*}
		\Th^k L\Th(X_A^i)=L\Th^{k+1}(X_A^i)+\mathscr{P}_{k+1,k}\Th^{k+1} X_A^i+\mathscr{P}_{k+1,k},
	\end{align*}
	where the terms $\mathscr{P}_{k+1,k}$ have already been determined on $C_0$. 
	
	Combining the above analysis, for $i=1,2,3$, we obtain an ODE system
	\[
	L\Th^{k+1}(X_A^i)=\sum_{j=1}^3\mathscr{P}_{k+2,k+1}\Th^{k+1}( X_A^j)+\mathscr{P}_{k+2,k+1},
	\]
	where he terms $\mathscr{P}_{k+2,k+1}$ have already been determined.  Since the value of $\Th^{k+1}(X_A^i) (i=1,2,3)$ on $S_{0,0}$ is determined by the initial data $(\mathring{\rho},\mathring{v},\mathring{s})$ on $\Sigma_0$, we can solve the above ODE system to obtain the value of $\Th^{k+1}(X_A^i)$ on $C_0$. Similar to Remark \ref{rem: local existence 4}, the ODE system can be solved on the entire time interval $[0,T]$.

	\subsubsection{Determine the remaining ${\rm ord}_T = k+1$ terms}
	We have already obtained the values of $\mathscr{P}_{k+1,k+1}$ terms for all $k\geqslant 0$. We now show that this  determines the value of all $\mathscr{P}_{n,k+1}$ terms for all $n\geqslant 1$ on $C_0$.

	By the definition of  $\mathscr{P}_{n,k+1}$, it suffices to consider a term $y$ of the form
	\[y=(\mathfrak{d}_1\circ\mathfrak{d}_2\circ\cdots \circ\mathfrak{d}_n)(x),\]
	where $\mathfrak{d}_i \in \{L,X_1,X_2,\Th\}$, with the restriction that
	\[\begin{cases}  \text{If }\ x\in \mathfrak{X}_0- \{\kappa,\kappa^{-1}\}, \ \  \big|\{i\big|\mathfrak{d}_i=\Th\}\big|=k+1;\\
		\text{If }~x\in \{\kappa,\kappa^{-1}\},  \ \  \big|\{i\big|\mathfrak{d}_i=\Th\}\big|=k.
	\end{cases}
	\]
	The idea of the proof is to use the commutator formula to move all $\Th$'s together. Let's take $x\in \mathfrak{X}_0- \{\kappa,\kappa^{-1}\}$ as an example. We now run an induction argument on $n$.  Assume that $\mathfrak{d}_{j-1}=\Th$, we use \eqref{eq: commutator schematic} to commute $\mathfrak{d}_{j-1}$ and $\mathfrak{d}_{j}$:
	\begin{align*}
		y&=(\mathfrak{d}_1\circ\cdots \circ \mathfrak{d}_{j-2}\circ \Th \circ\mathfrak{d}_{j} \circ \cdots \circ\mathfrak{d}_n)(x)\\
		&=(\mathfrak{d}_1\circ\cdots \circ \mathfrak{d}_{j-2}\circ \big(\mathscr{P}_{1,1} \cdot\mathfrak{d}\big)\circ \cdots \circ\mathfrak{d}_n)(x)
		+(\mathfrak{d}_1\circ\cdots \circ \mathfrak{d}_{j-2}\circ\mathfrak{d}_{j} \circ \Th \circ \cdots \circ\mathfrak{d}_n)(x).
	\end{align*}
	For the first term, the total number of derivatives is $n-1$ and it is determined by the inductive hypothesis. Therefore, it suffices to determine the second term. We can repeat the above argument  and reduce the problem to determine $y$ in the following form:
	\begin{align*}
		y&=(\mathfrak{d}_1\circ\cdots \circ \mathfrak{d}_{n-k-1}\circ \Th^{k+1})(x).
	\end{align*}
	where the $\mathfrak{d}_i$'s are from $\{L,X_1,X_2\}$. Since $\Th^{k+1}(x)$ is already determine on $C_0$, we can determine the value of $y$ on $C_0$.  
	
	This completes the proof of the main result (Theorem \ref{thm:main theorem}) of the paper.
	
	\bibliography{bib.bib}

\begin{thebibliography}{10}

\bibitem{Abbrescia-SpeckCauchyHorizon}
Leonardo Abbrescia and Jared Speck.
\newblock The emergence of the {C}auchy horizon from the crease in 3{D}
  compressible {E}uler flow.
\newblock in preparation.

\bibitem{Abbrescia-Speck22}
Leonardo Abbrescia and Jared Speck.
\newblock The emergence of the singular boundary from the crease in $3{D}$
  compressible {E}uler flow, 2022.
\newblock arXiv:2207.07107.

\bibitem{Abbrescia-Speck23}
Leonardo Abbrescia and Jared Speck.
\newblock The relativistic euler equations: Esi notes on their geo-analytic
  structures and implications for shocks in 1d and multi-dimensions.
\newblock {\em Classical and Quantum Gravity}, 40(24):243001, nov 2023.

\bibitem{C07}
Demetrios Christodoulou.
\newblock {\em The formation of shocks in 3-dimensional fluids}.
\newblock EMS Monographs in Mathematics. European Mathematical Society (EMS),
  Z\"{u}rich, 2007.

\bibitem{C09}
Demetrios Christodoulou.
\newblock {\em The formation of black holes in general relativity}.
\newblock EMS Monographs in Mathematics. European Mathematical Society (EMS),
  Z\"{u}rich, 2009.

\bibitem{C19}
Demetrios Christodoulou.
\newblock {\em The shock development problem}.
\newblock EMS Monographs in Mathematics. European Mathematical Society (EMS),
  Z\"{u}rich, 2019.

\bibitem{C-K93}
Demetrios Christodoulou and Sergiu Klainerman.
\newblock {\em The global nonlinear stability of the {M}inkowski space},
  volume~41 of {\em Princeton Mathematical Series}.
\newblock Princeton University Press, Princeton, NJ, 1993.

\bibitem{C-Miao14}
Demetrios Christodoulou and Shuang Miao.
\newblock {\em Compressible flow and {E}uler's equations}, volume~9 of {\em
  Surveys of Modern Mathematics}.
\newblock International Press, Somerville, MA; Higher Education Press, Beijing,
  2014.

\bibitem{Disconzi-Speck19}
Marcelo~M. Disconzi and Jared Speck.
\newblock The relativistic {E}uler equations: remarkable null structures and
  regularity properties.
\newblock {\em Ann. Henri Poincar\'{e}}, 20(7):2173--2270, 2019.

\bibitem{Euler1757}
Leonhard Euler.
\newblock Principes généraux du mouvement des fluides.
\newblock {\em Mémoires de l'académie des sciences de Berlin}, 11:274--315,
  1757.

\bibitem{CharivpLisibach}
Andr\'{e} Lisibach.
\newblock Characteristic initial value problem for spherically symmetric
  barotropic flow.
\newblock {\em J. Hyperbolic Differ. Equ.}, 14(4):565--589, 2017.

\bibitem{lisibach2021shock}
Andr\'{e} Lisibach.
\newblock Shock reflection in plane symmetry, 2021.
\newblock arXiv:2112.15266.

\bibitem{lisibach2022shock}
Andr\'{e} Lisibach.
\newblock Shock interaction in plane symmetry, 2022.
\newblock arXiv:2202.08111.

\bibitem{Luk12}
Jonathan Luk.
\newblock On the local existence for the characteristic initial value problem
  in general relativity.
\newblock {\em Int. Math. Res. Not. IMRN}, (20):4625--4678, 2012.

\bibitem{Luk-Speck20}
Jonathan Luk and Jared Speck.
\newblock The hidden null structure of the compressible {E}uler equations and a
  prelude to applications.
\newblock {\em J. Hyperbolic Differ. Equ.}, 17(1):1--60, 2020.

\bibitem{Rendall90}
A.~D. Rendall.
\newblock Reduction of the characteristic initial value problem to the {C}auchy
  problem and its applications to the {E}instein equations.
\newblock {\em Proc. Roy. Soc. London Ser. A}, 427(1872):221--239, 1990.

\bibitem{Riemann1860}
Bernhard Riemann.
\newblock {\"U}ber die fortpflanzung ebener luftwellen von endlicher
  schwingungsweite.
\newblock {\em Abhandlungen der Königlichen Gesellschaft der Wissenschaften in
  Göttingen}, 8:43--66, 1860.

\bibitem{Speck19}
Jared Speck.
\newblock A new formulation of the 3{D} compressible {E}uler equations with
  dynamic entropy: remarkable null structures and regularity properties.
\newblock {\em Arch. Ration. Mech. Anal.}, 234(3):1223--1279, 2019.

\bibitem{Speck-Yu}
Jared Speck and Sifan Yu.
\newblock Characteristic initial value problem for the 3{D} compressible
  {E}uler equations.
\newblock in prep.

\bibitem{wang2023shock}
Yuxuan Wang.
\newblock Shock interaction in sphere symmetry, 2023.
\newblock arXiv:2310.06510.

\end{thebibliography}
	\bibliographystyle{plain}

\end{document}